\documentclass[12pt]{article}
\usepackage{e-jc}

\usepackage{amsthm,amsmath,amssymb}
\usepackage{appendix}
\usepackage{graphicx}

\usepackage[colorlinks=true,citecolor=black,linkcolor=black,urlcolor=blue]{hyperref}
 \usepackage{tikz}
\theoremstyle{plain}
\newtheorem{theorem}{Theorem}
\newtheorem{lemma}[theorem]{Lemma}
\newtheorem{corollary}[theorem]{Corollary}

\theoremstyle{definition}

\newtheorem{conjecture}[theorem]{Conjecture}

\theoremstyle{remark}

\newcommand{\ZZ}{\mathbb Z}
\newcommand{\NN}{\mathbb N}
\newcommand{\eqn}[1]{(\ref{#1})}

\newcommand{\al}{\alpha} 
\newcommand{\THETA}{\varTheta} 

\newcommand{\seqnum}[1]{\href{https://oeis.org/#1}{\underline{#1}}} 
\newcommand{\mex}{\mbox{mex}} 
\newcommand{\sP}{{\mathcal P}} 
\newcommand{\TTW}{{\bf T}} 

\newcommand*\circled[1]{\tikz[baseline=(char.base)]{
            \node[shape=circle,draw,inner sep=2pt] (char) {#1};}}

\newcommand{\eeq}{\end{equation}}
\newcommand{\beql}[1]{\begin{equation}\label{#1}}

\title{\bf Queens in exile: non-attacking queens on infinite chess boards}

\author{F. Michel Dekking\\
\small Applied Mathematics (DIAM), Delft University of Technology\\[-0.8ex]
\small 2600 GA Delft, The Netherlands\\
\small\tt F.M.Dekking@math.tudelft.nl\\
\and
Jeffrey Shallit\\
\small School of Computer Science, University of Waterloo\\[-0.8ex]
\small Waterloo, ON N2L 3G1, Canada\\
\small\tt shallit@waterloo.ca\\
\and
N. J. A. Sloane\footnote{Corresponding author} \\
\small The OEIS Foundation Inc.\\[-0.8ex]
\small 11 So. Adelaide Ave., Highland Park, NJ 08904, U.S.A. \\
\small\tt njasloane@gmail.com\\
}

\date{\dateline{July 27, 2019}{}\\
\small  Mathematics Subject Classification MSC2010: 91A46}

\begin{document}

\maketitle


\begin{abstract}


Number the cells of a (possibly infinite) chessboard in some way
with the numbers $0, 1, 2, \ldots\,$. Consider the cells in order, placing a queen
in a cell  if and only if it would not attack any earlier queen.    The problem is to determine 
the positions of the queens.  We study the problem for  
a doubly-infinite chessboard of size $\ZZ \times \ZZ$ numbered along a square spiral, and
an infinite single-quadrant chessboard (of size
$\NN \times \NN$) numbered along antidiagonals.
We give a fairly complete solution in the first case, based on the Tribonacci word.
There are  connections with combinatorial games.
 
\bigskip\noindent \textbf{Keywords:}
Tribonacci word, Tribonacci representation, Greedy Queens, Wythoff Nim, combinatorial games, Sprague-Grundy function

\vspace{0.1in}

\end{abstract}


\section{Queens in exile}\label{Sec1}
\begin{figure}[!ht]
\centerline{\includegraphics[angle=0, width=5in]{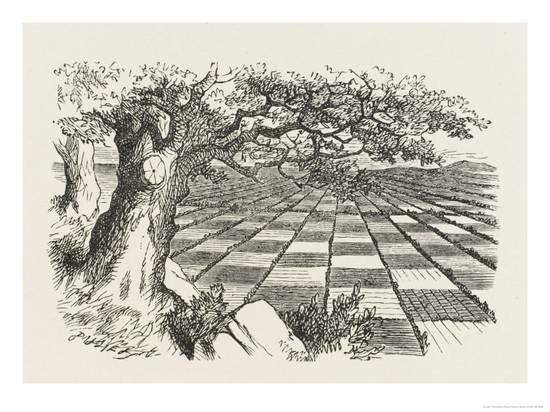}}
\caption{The great plain of Attak\'{i}a. [John Tenniel, Illustration for Lewis Carroll, 
\emph{Through the looking-glass and what Alice found there} (1871).] }
\label{Fig1}
\end{figure}

The rival queens in the mythical country of Attak\'{i}a 
have been quarreling, and have agreed to go into exile. 
The great plain has been divided into squares, which have been numbered 
in a square spiral (Figs.~\ref{Fig1}, \ref{Fig2}).
The first queen settles at square $0$. The next queen proceeds along the square spiral 
and settles at the first square she reaches from which she cannot attack the first queen: this is square $9$.
The process is repeated for all the queens. Each queen settles at
the first square along the spiral from which 
she cannot attack  any queen who is already settled. 
The positions of the first nine queens  are indicated by circles in Fig.~\ref{Fig2}.
The squares along the spiral where they settle form the sequence
\beql{A273059}
0, 9, 13, 17, 21, 82, 92, 102, 112, 228, 244, 260, 276, 445, 467, 489, 511, 630, \ldots
\eeq
(\seqnum{A273059}\footnote{Six-digit numbers prefixed by A refer to entries in the \emph{On-Line Encyclopedia of Integer Sequences} \cite{OEIS}.}
 in \cite{OEIS}).

\begin{figure}[!ht]
\centerline{\includegraphics[angle=0, width=7in]{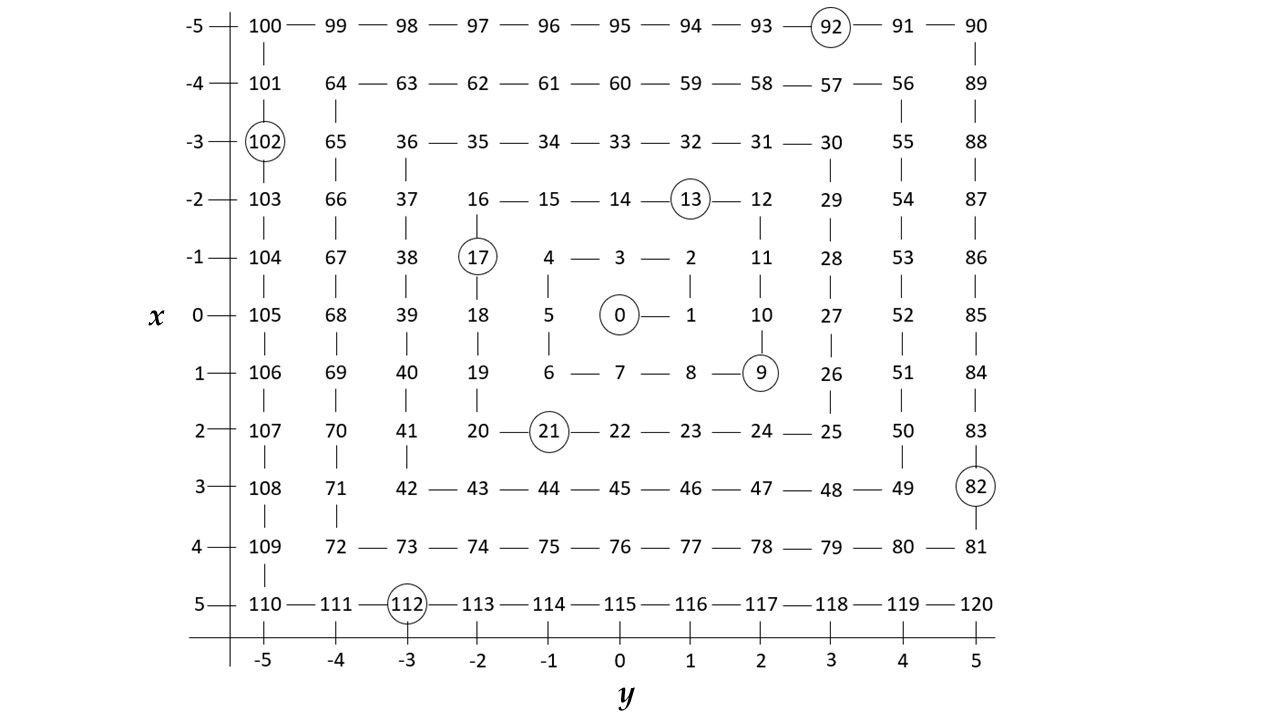}}
\caption{Squares of a doubly-infinite chessboard numbered along a square spiral.
Positions of the first nine exiled queens are circled. [Figure courtesy of Jessica Gonzalez.]}
\label{Fig2}
\end{figure}

\begin{figure}[!ht]
\centerline{\includegraphics[angle=0, width=4in]{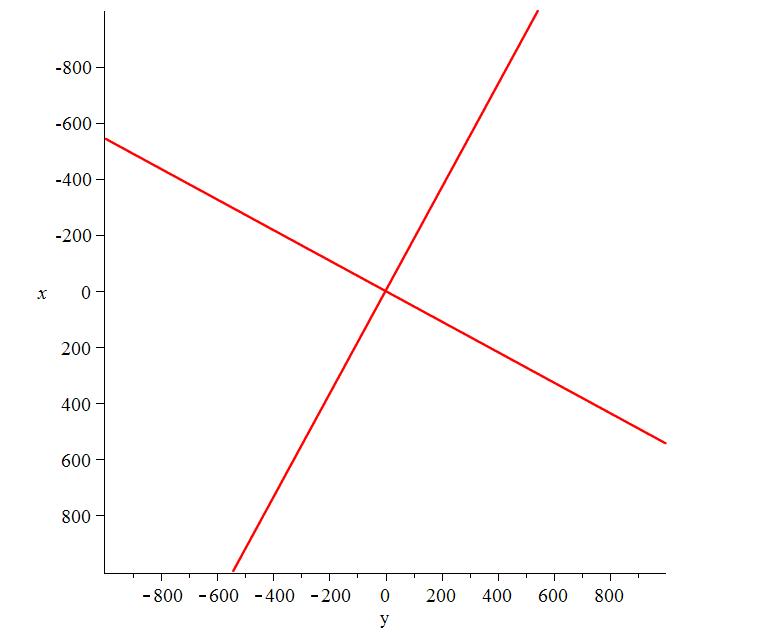}}
\caption{Positions of the first 1409 queens (those with maximum coordinate in the range $-1000$ to $1000$).
At this scale the points lie essentially on four straight lines. [Figure courtesy of Alois Heinz.]}
\label{Fig3}
\end{figure}

Figure~\ref{Fig3} shows the positions of the first $1409$ queens.  At this scale one
can see that the points lie essentially on four straight lines, and that the configuration
has cyclic four-fold symmetry. The main goals of the first part of the
paper are to determine the positions of the queens, to establish the
cyclic symmetry, and to show that the
slopes of the four lines are $\pm\psi$ and $\pm1/\psi$,
where $\psi \approx 1.8393$ is the Tribonacci constant, the  real root of $x^3=x^2+x+1$.
These results are established in Sections~\ref{SecQOSS}--\ref{SecMain}.

In Section~\ref{SecQOSS} we first show that the positions of the queens are determined by 
certain recursively defined quadruples of integers $X_n$, $Y_n$, $M_n$, $P_n$, $n \ge 0$
(see \eqn{EqXYMP0}, Tables~\ref{TabXYMP}, \ref{TabXYMP2},  and Theorem~\ref{ThXYMP}).
A study of the first differences $\Delta X_n$, etc., of the $X_n$, $Y_n$, $M_n$, $P_n$
sequences suggests that all four can be defined in terms of a certain
three-letter sequence that we call the ``theme song'', and denote
by $\THETA(a,b,c)$ (Section~\ref{SecTheme}). 

Rather surprisingly, the theme song turns out to be a disguised version of the classic
three-letter Tribonacci word $\TTW (a,b,c)$ (see Theorem~\ref{Lem3} in Section~\ref{SecLem3}).  
Section~\ref{SecProps} contains a number of properties of the Tribonacci word that will be used 
later.  Some of these properties appear to be new 
(Theorems~\ref{thm13}--\ref{thm17}, for example), although it is difficult to be certain
because so much has already been published about the Tribonacci word.

In Section~\ref{SecMain} we
establish our main theorem, Theorem~\ref{ThMT1}, which shows
that the rows of the $XYMP$ table are in one-to-one
correspondence with the terms of the theme song $\THETA$ (or,
if we ignore the initial $n=0$ term,  with the terms of the Tribonacci word $\TTW$).
Corollary~\ref{CorXYABC} establishes some unexpected
connections between the $X_n$, $Y_n$, $M_n$, $P_n$ sequences
and the $A_n$, $B_n$, $C_n$ sequences studied in \cite{BBB04, CSH72, DuRi08}
and in \S\ref{Sec5_4}.
Remark~(iv) following Theorem~\ref{ThMT1} shows that the slopes of the lines
containing the queens are as claimed.  

The reader may wonder why we use both the theme song $\THETA(a,b,c)$ and the Tribonacci word
$\TTW (a,b,c)$, when these sequences are so similar. The answer is that 
we need $\TTW(a,b,c)$ because so much is known about its properties (see Section~\ref{SecProps}), 
whereas $\THETA(a,b,c)$ is more
in tune with the $XYMP$ table, since  the lengths $6$, $5$, $4$ of the images $\theta(a)$,
$\theta(b)$, $\theta(c)$ (see~\eqn{Eqtheta}) match the block structure of the table,
as can be seen by comparing \eqn{EqDP} and \eqn{EqDX} with \eqn{EqTheme}.

In Section~\ref{SecSG} we consider the same problem in the setting of
combinatorial games. The positions of the queens are the $\sP$-positions
in a certain game, and so correspond to the $0$ entries in the 
table of Sprague-Grundy values for the game (see Fig.~\ref{Fig4} below).
Although we have been able to
determine the positions of the queens, we have not been able to answer
a natural question about the Sprague-Grundy values: are all the rows, columns, 
and diagonals of Fig.~\ref{Fig4} permutations of the
nonnegative integers  (Conjecture~\ref{Conj1})?

Similar questions can be asked for chessboards of other shapes. 
The general setting for the problem is that the cells of the board
are numbered in some way with the numbers $0,1,2,3,\ldots\,$. 
We consider the cells in order, placing a chess queen in cell $n$ if and only if
it would not attack any earlier queen.  The problem is
to determine the positions of the queens.

In Section~\ref{SecQ1} we consider the case of an infinite $\NN \times \NN$ board (that is,
a single-quadrant board), where the squares are numbered along successive 
antidiagonals, as shown in Table~\ref{FigQ1a}.
For this version of the problem, the data shows  overwhelmingly  that 
the queens lie essentially on two  straight lies,
of slopes  $\phi$ and $1/\phi$, where $\phi$ is the golden ratio.
It is regrettable that we have not been able to prove this.
On the other hand, we have been able to prove
that all the rows and columns of the Sprague-Grundy table
are  permutations of the nonnegative integers 
(although not that the diagonals are).
So for this problem, our results are both weaker and stronger than for the 
queens-on-a-square-spiral problem.

\vspace*{+.1in}
Many other examples can be found in \cite{OEIS}.
These include:
\begin{itemize}
\item the board formed from a $45$-degree sector of a single quadrant 
(cells  $\{(x,y) \in \NN \times \NN:x \ge y\}$)  (\seqnum{A274650}),
\item  finite boards of size $n \times n$  (\seqnum{A308880}, \seqnum{A308881}),
\item  boards with hexagonal cells (\seqnum{A274820}, \seqnum{A296339}),
\item  one may also ask similar questions using other chess pieces instead of queens:
kings (\seqnum{A275609}),
knights (\seqnum{A308884}),
rooks (\seqnum{A308896}),  
or Maharajas (pieces that combine the moves of a queen and a knight \cite{LaWa14}: \seqnum{A307282}).
\end{itemize}
\noindent
Much is known about these examples, but there is no space to discuss them here. There are many open questions. 

\vspace*{+.2in}
\noindent{\bf Historical remarks.}
To the best of our knowledge, the first mention of 
any of these queens-in-exile problems was in
\cite{OEIS}, in October 2001, when
Antti Karttunen contributed  \seqnum{A065188},
a version of the single-quadrant sequence \seqnum{A275895}. 
He called it a ``Greedy Queens'' sequence
(referring to the fact that
the queens are placed using the greedy algorithm---no disrespect
to the queens was intended).
The problem on the $\ZZ \times \ZZ$ board stated at the beginning of this article was introduced
(using somewhat different language) by Paul~D.~Hanna in June 2008
when he submitted  \seqnum{A140100}--\seqnum{A140103} to~\cite{OEIS}
(these are the $X_n, Y_n, M_n, P_n$ sequences), and implicitly stated what is now Theorem~\ref{ThXYMP} below. 
The Sprague-Grundy values for the single-quadrant version were contributed by Alec Jones
in April 2016, in \seqnum{A269526}.
The connections between the exiled queens problems  and combinatorial games 
were pointed out by  Allan C. Wechsler in a comment on \seqnum{A274528}.
The Sprague-Grundy values for the $\ZZ \times \ZZ$ board 
numbered along a square spiral originated in \seqnum{A274640},
contributed in June 2016 by Zak Seidov and Kerry Mitchell.
Since then, a large number of other authors (too many to mention here)
have added further sequences of this type, or contributed comments, computer
programs, additional terms, etc. 

Very recently, Fokkink and Rust introduced in \cite{FR19} a two-pile combinatorial game they 
call \emph{Splythoff}, where the $\sP$-positions are given by the queen positions $(X_n,Y_n)$. 
This is another variant of Wythoff's Nim. It is different from the game
we discuss in Section~\ref{SecSG}, since their piles contain only nonnegative numbers of tokens.

\vspace*{+.2in}
\noindent{\bf Notation.}
The ternary Tribonacci word will be denoted by $\TTW = t_1 t_2 t_3 \cdots$, or by $\TTW (a,b,c)$ when
we wish to emphasize which three-letter alphabet is being used.
$\psi = 1.839286755214\ldots$ is the Tribonacci constant, the real root of 
$x^3-x^2-x-1$, and 
$\phi$ is the golden ratio.
$|S|$ denotes the cardinality of a set, the length of a word, or the absolute value of a complex number.
For a set $S$, $\mex (S)$ is the {\em minimum excluded value},
that is, the smallest nonnegative number not in $S$  \cite{Guy91}. 
For a sequence $\{s_n\}$, the difference operator is defined by 
$\Delta s_n := s_{n+1}-s_{n}$.
We use a centered dot ($\cdot$) to indicate concatenation of words
(or, rarely, the product of two numbers).
$\ZZ$ and $\NN$ are the integers and nonnegative integers, respectively.
For any undefined terms from combinatorial games or combinatorics on words, see
\cite{AlSh03, Guy91, Loth83}.


\section{Queens on a square spiral}\label{SecQOSS}

In this section we study the problem on a doubly-infinite chessboard.
The cells are unit squares centered at the points of a $\ZZ \times \ZZ$ grid.
We construct a ``square spiral'' by starting at the central square and
proceeding counter-clockwise, moving successively 
East, North, West, South, East, North, ... The cells are numbered $0,1,2,3,\ldots$ (see Fig.~\ref{Fig2}).
The exiled queens are placed according to the rule specified in the
opening paragraph of the previous section.

We take the $x$-axis to point South and the $y$-axis
to point East, as shown
in the coordinate axes in Fig.~\ref{Fig2}. This puts the main line of queens 
(the queens in cells $0$, $9$, $82$, $228$, $445$, $630, \ldots$ in the first quadrant,
and is also consistent with having the origin for
the single-quadrant version of the problem (Section~\ref{SecQ1}) 
in the top left corner of the board,
as in the discussions of the related games Wyt Queens and  Wythoff's Nim  in \cite{WW}.

We consider the square spiral as being built up from a  
series of square ``shells''.  Shell $0$ is the starting cell at the center.
Shell $k$ ($k =1,2,\ldots)$ consists of the $8k$ cells
labeled $(2k-1)^2$ to $4k(k+1)$. The spiral
traverses shells $0,1,2,\ldots$ in order.
Shell $k$ has four edges, each containing $2k$ cells.
Edge $1$ (on the right) consists of cells $(2k-1)^2$ through $4k^2-2k$,
edge $2$ (at top): cells $4k^2-2k+1$ through $4k^2$,
edge $3$ (on left): cells $4k^2+1$ through $4k^2+2k$,
and edge $4$ (at bottom): cells $4k^2+2k+1$ through $4k^2+4k$.
The spiral traverses shell $k$ along successive edges $1,2,3,4$.

We see that the cyclic group of order $4$ generated
by $(x,y) \mapsto (-y,x)$ preserves the points in each shell.

\begin{table}[htb]
\caption{ Initial values of $X_n, Y_n, M_n, P_n$. }
\label{TabXYMP}
$$
\begin{array}{|c|cc|cc|}
\hline
n & X_n & Y_n & M_n & P_n \\
\hline
0  &  0 & 0 & 0 & 0 \\
1  &  1 & 2 & 1 & 3 \\
2  &  3 & 5 & 2 & 8 \\
3  &  4 & 8 & 4 & 12 \\
4  &  6 & 11 & 5 & 17 \\
5  &  7 & 13 & 6 & 20 \\
6  &  9 & 16 & 7 & 25 \\
7  & 10 & 19 & 9 & 29 \\
8  & 12 & 22 & 10 & 34\\
9  & 14 & 25 & 11 & 39 \\
10 & 15 & 28 & 13 & 43 \\
11 & 17 & 31 & 14 & 48 \\
\hline
\end{array}
$$
\end{table}

As Paul Hanna realized in 2008, the positions of the queens in the spiral are determined by certain quadruples 
of nonnegative integers $X_n$, $Y_n$, $M_n$, $P_n$ 
 $(n \ge 0)$, defined by $X_0 = Y_0 = M_0 = P_0 = 0$
 and, for $n>0$, 
\begin{align}\label{EqXYMP0}
X_n &~=~  \mex \{ X_i, Y_i : i<n\}, \nonumber \\
M_n &~=~  \mex \{ M_i, P_i : i<n\}, \nonumber \\
Y_n & ~=~ X_n + M_n, \nonumber \\
P_n & ~=~ X_n + Y_n, 
\end{align}
where $\mex$ denotes ``minimum excluded value'' as defined above.
The initial values of these quadruples are shown in Table~\ref{TabXYMP}
(the ``XYMP table''),
and a more extensive list is given  in Table~\ref{TabXYMP2} below.
These are Paul Hanna's sequences \seqnum{A140100}--\seqnum{A140103}.


The following properties are immediate consequences of the definition:
$\{X_n\}$ and $\{Y_n\}$ are a pair of complementary sequences,
as are $\{M_n\}$ and $\{P_n\}$. 
All four sequences are monotonically
increasing, so $\Delta X_n \ge 1$, $\Delta M_n \ge 1$,
$\Delta Y_n \ge 2$, $\Delta P_n \ge 3$. 
Also $\Delta X_n \le 2$ (if $\Delta X_n=3$ there would be a pair of adjacent $Y$ values
differing by $1$, contradicting $\Delta Y_n \ge 2$).
Similarly $\Delta M_n \le 2$, $\Delta Y_n \le 4$,
$\Delta P_n \le 6$.
(In fact $\Delta Y_n$ is never $4$  and $\Delta P_n$ is never $6$,
although we will not prove this until Section~\ref{SecMain}.)
Also $Y_n \ge X_n+1$ and $P_n \ge M_n+2$ for $n>0$.
 
Let $q_n\, (n \ge 0)$ denote the $(x,y)$ coordinates of the $n$th queen
in the spiral. We saw in Fig.~\ref{Fig2} that 
$q_0=(0,0)$, $q_1=(1,2)$, $q_2=(-2,1)$, $q_3=(-1,-2)$, $q_4=(2,-1)$, $q_5=(3,5), \ldots\,$.

The following theorem is implicit in Paul Hanna's remarks in \seqnum{A140100}--\seqnum{A140103}.

\begin{theorem}\label{ThXYMP}
After the initial queen is placed at $q_0=(X_0, Y_0)$, the subsequent queens 
are placed at
\beql{EqXYMP1}
q_{4k+1} = (X_k,Y_k),  ~q_{4k+2} = (-Y_k, X_k), 
~q_{4k+3} = (-X_k, -Y_k),
~q_{4k+4} = (-Y_k, X_k),
\eeq
for $k=0,1,2,\ldots\,.$
\end{theorem}

\begin{proof}
Note that the points \eqn{EqXYMP1} lie on shell $Y_k$ of the spiral,
and this set of four points  is preserved by the cyclic group of order $4$.
We establish \eqn{EqXYMP1} by induction on $k$. The result is true for $k=1$.

Suppose  that the hypothesis holds for $k=0,1, \ldots, n$.
Call a square ``free'' if a queen at that square
would not attack any existing queen.
After $q_0, \ldots, q_{4n+4}$ have
been placed, a square $(x,y)$ is not free if any of the following hold:
  \begin{itemize}
  \item $x$ is equal to $\pm X_i$ or $\pm Y_i$ for some $0 \le i \le n$, 
  \item $y$ is equal to $\pm X_i$ or $\pm Y_i$ for some $0 \le i \le n$, 
  \item $y-x$ is equal to $\pm M_i$ or $\pm P_i$ for some $0 \le i \le n$, 
  \item $y+x$ is equal to $\pm M_i$ or $\pm P_i$ for some $0 \le i \le n$,
  \end{itemize}
because $(x,y)$ would then be on the same row, column, or diagonal
as one of the existing queens.

When we move along the spiral
after placing $q_{4n+4}$, the first
square we reach that does not satisfy any of
these conditions is (by \eqn{EqXYMP0})   $(X_{n+1}, Y_{n+1})$, which is therefore $q_{4n+5}$.
Since $0 < X_{n+1} < Y_{n+1}$, this lies on edge $1$ of shell $Y_{n+1}$.

As we continue around the spiral, we next reach edge 2 of the same shell.
Since the configuration of existing queens is preserved by the cyclic group,
we would have $q_{4n+6} = (-Y_{n+1}, X_{n+1})$, except we must check that this
square does not attack the queen $q_{4n+5}$ we  just placed.
However, the line from $(-Y_{n+1}, X_{n+1})$ to $(X_{n+1}, Y_{n+1})$
has slope $(Y_{n+1}-X_{n+1})/(Y_{n+1}+X_{n+1})$, which is not $\pm 1$,
since neither $X_{n+1}$ nor $Y_{n+1}$ is $0$. So $q_{4n+6} = (-Y_{n+1}, X_{n+1})$.

Similar arguments show that $q_{4n+7}=(-X_{n+1}, -Y_{n+1})$ and 
$q_{4n+8}=(Y_{n+1}, -X_{n+1})$. Thus \eqn{EqXYMP1} holds for $k=n+1$.
 \end{proof}


\begin{table}[htb]
\caption{ The sequences $X_n$, $Y_n$, $M_n$, $P_n$ and their differences, the identification of the rows
with the ``theme song'' $\THETA(a,b,c) = \{t_n: n \ge 0\}$, 
and the sequences $A_n$, $B_n$, $C_n$. }
\label{TabXYMP2}
\footnotesize
$$
\begin{array}{|c|c||cc|cc||cc|cc||ccc|}
\hline
n & t_n & X_n & Y_n & M_n & P_n & \Delta X_n & \Delta Y_n & \Delta M_n & \Delta P_n & A_n & B_n & C_n \\
\hline
0 & c & 0 & 0 & 0 & 0 & 1 & 2 & 1 & 3 & 0 & 0 & 0 \\
\hline
1 & a & 1 & 2 & 1 & 3 & 2 & 3 & 1 & 5 & 1 & 2 & 4 \\
2 & b & 3 & 5 & 2 & 8 & 1 & 3 & 2 & 4 & 3 & 6 & 11 \\
3 & a & 4 & 8 & 4 & 12 & 2 & 3 & 1 & 5 & 5 & 9 & 17 \\
4 & c & 6 & 11 & 5 & 17 & 1 & 2 & 1 & 3 & 7 & 13 & 24 \\
\hline
5 & a & 7 & 13 & 6 & 20 & 2 & 3 & 1 & 5 & 8 & 15 & 28 \\
6 & b & 9 & 16 & 7 & 25 & 1 & 3 & 2 & 4 & 10 & 19 & 35 \\
7 & a & 10 & 19 & 9 & 29 & 2 & 3 & 1 & 5 & 12 & 22 & 41 \\
8 & a & 12 & 22 & 10 & 34 & 2 & 3 & 1 & 5 & 14 & 26 & 48 \\
9 & b & 14 & 25 & 11 & 39 & 1 & 3 & 2 & 4 & 16 & 30 & 55 \\
10 & a & 15 & 28 & 13 & 43 & 2 & 3 & 1 & 5 & 18 & 33 & 61 \\
11 & c & 17 & 31 & 14 & 48 & 1 & 2 & 1 & 3 & 20 & 37 & 68 \\
\hline
12 & a & 18 & 33 & 15 & 51 & 2 & 3 & 1 & 5 & 21 & 39 & 72 \\
13 & b & 20 & 36 & 16 & 56 & 1 & 3 & 2 & 4 & 23 & 43 & 79 \\
14 & a & 21 & 39 & 18 & 60 & 2 & 3 & 1 & 5 & 25 & 46 & 85 \\
15 & b & 23 & 42 & 19 & 65 & 1 & 3 & 2 & 4 & 27 & 50 & 92 \\
16 & a & 24 & 45 & 21 & 69 & 2 & 3 & 1 & 5 & 29 & 53 & 98 \\
17 & c & 26 & 48 & 22 & 74 & 1 & 2 & 1 & 3 & 31 & 57 & 105 \\
\hline
18 & a & 27 & 50 & 23 & 77 & 2 & 3 & 1 & 5 & 32 & 59 & 109 \\
19 & b & 29 & 53 & 24 & 82 & 1 & 3 & 2 & 4 & 34 & 63 & 116 \\
20 & a & 30 & 56 & 26 & 86 & 2 & 3 & 1 & 5 & 36 & 66 & 122 \\
21 & a & 32 & 59 & 27 & 91 & 2 & 3 & 1 & 5 & 38 & 70 & 129 \\
22 & b & 34 & 62 & 28 & 96 & 1 & 3 & 2 & 4 & 40 & 74 & 136 \\
23 & a & 35 & 65 & 30 & 100 & 2 & 3 & 1 & 5 & 42 & 77 & 142 \\
24 & c & 37 & 68 & 31 & 105 & 1 & 2 & 1 & 3 & 44 & 81 & 149 \\
\hline
25 & a & 38 & 70 & 32 & 108 & 2 & 3 & 1 & 5 & 45 & 83 & 153 \\
26 & b & 40 & 73 & 33 & 113 & 1 & 3 & 2 & 4 & 47 & 87 & 160 \\
27 & a & 41 & 76 & 35 & 117 & 2 & 3 & 1 & 5 & 49 & 90 & 166 \\
28 & c & 43 & 79 & 36 & 122 & 1 & 2 & 1 & 3 & 51 & 94 & 173 \\
\hline 
\end{array}
$$
\end{table}
\normalsize


\section{The ``theme song''}\label{SecTheme}

Although it is not immediately apparent, all four columns of the XYMP table
are variations on a single sequence.  This ``theme song'' is
most visible when we examine the differences $\{\Delta P_n\}$
of the $P_n$ column of the table, keeping in mind the observations about 
these differences that were made following \eqn{EqXYMP0}.
The differences $\{\Delta P_n\}$ begin
\beql{EqDP}
3, 5, 4, 5, ~~3, 5, 4, 5, 5, 4, 5, ~~3, 5, 4, 5, 4, 5, ~~3, 5, 4, 5, 5, 4, 5, ~~3, 5,4,5, ~~3,  \ldots\, ,
\eeq
where we have inserted spaces to highlight the block structure. The differences
of the other columns show a similar, although less obvious, structure.
For example, the differences $\{\Delta X_n\}$ begin
\beql{EqDX}
1, 2, 1, 2, ~~1, 2, 1, 2, 2, 1, 2, ~~1, 2, 1, 2, 1, 2, ~~1, 2, 1, 2, 2, 1, 2, ~~1, 2, 1, 2, ~~1, \ldots\, ,
\eeq
where we have used the same block lengths.
As we will prove in Theorem~\ref{ThMT1}, all four column differences are instances
of the sequence $\THETA = \THETA(a,b,c)$ (the ``theme song''),  the fixed
point of the morphism $\theta$ defined over the
alphabet $\{a,b,c\}$ by
\beql{Eqtheta}
\theta:  ~~a \to cabaaba,  ~~b \to cababa,  ~~c \to caba\,.
\eeq
$\THETA(a,b,c)$ begins
\beql{EqTheme}
c,a,b,a, ~~c,a,b,a,a,b,a, ~~c,a,b,a,b,a, ~~c,a,b,a,a,b,a, ~~c,\ldots\, ,
\eeq
and, as we will see, $\{\Delta P_n\}= \THETA(5,4,3)$
and $\{\Delta X_n\} = \THETA(2,1,1)$.


\section{The Tribonacci word}\label{SecLem3}
It was a further surprise to discover that  $\THETA(a,b,c)$ is itself a 
lightly disguised version of the classic Tribonacci word.
The {\em Tribonacci word} $\TTW = \TTW (a,b,c) = \{t_n: n \ge 1\}$ is 
the fixed point of the {\em Tribonacci morphism}
\beql{Eqtau}
\tau: ~~a \to ab, ~~ b \to ac,  ~~ c \to a \,.
\eeq
$\TTW (a,b,c)$ begins
\beql{EqTabc}
a,b,a,c,a,b,a,a,b,a,c,a,b,a,b,a,c,a,b,a,a,b,a,c,a,\dots\,.
\eeq
There is an extensive literature--see for example  
\cite{BBB04, CSH72, DFG17,
DuRi08, MoSh14, TaWe07}, as well as the references cited in those papers.

\begin{theorem} \label{Lem3}
We have
\beql{EqLem3}
\THETA (a,b,c) ~=~ c \cdot \TTW (a,b,c).
\eeq
\end{theorem}

\begin{proof}
The morphism $\al := \tau^3$ maps
$$
a \to abacaba, ~~ b \to abacab, ~~ c \to abac \,.
$$
If the prefix $aba$ in these three images is changed to a suffix,
$\al$ becomes $\theta$; that is, for single letters $x$, 
$\al(x)\,aba = aba \, \theta(x)$. 
So $\al(x) = aba \,\theta(x) (aba)^{-1}$, and since $\al$ and $\theta$ are morphisms,  
$\al(w) = aba\,\theta(w) (aba)^{-1}$ for any word $w$, that is,
\beql{EqLem3a}
\al(w)\,aba ~=~ aba\,\theta(w)\,.
\eeq
We prove \eqn{EqLem3} by showing that, for all $k \ge 0$,
\beql{EqLem3b}
\theta^k(c)\,c ~=~ c\,\al^k(c)\,.
\eeq
We use induction on $k$. The result is certainly true for $k=0$ and $1$.
Suppose it holds for $k$, and set $w=\al^k(c)$ in \eqn{EqLem3a}.
We have
\begin{align}
c\, \al^{k+1}(c) \,aba & ~=~ c\,aba\,\theta(\al^k(c)), ~~  \text{by}~\eqn{EqLem3a}, \nonumber \\
   & ~=~ \theta(c) \theta(\al^k(c)) \nonumber \\
   & ~=~ \theta(c\,\al^k(c)) \nonumber \\
   & ~=~ \theta(\theta^k(c)\,c) ~\text{(by the induction hypothesis)} \nonumber \\
   & ~=~ \theta^{k+1}(c)\,caba\,, \nonumber
   \end{align}
   and canceling $aba$ from both sides we obtain $c\,\al^{k+1}(c) = \theta^{k+1}(c)\,c$,
   as required. Letting $k \to \infty$ in \eqn{EqLem3b} completes the proof of \eqn{EqLem3}. 
\end{proof}

In view of Theorem~\ref{Lem3} we define $t_0=c$, so that $\THETA(a,b,c) = \{t_n : n \ge 0\}$.


\section{Properties of the Tribonacci word}\label{SecProps}\label{Sec5}

The Tribonacci word $\TTW = \{ t_n: n \ge 1\}$ is an analog for a three-letter
alphabet of the even more classic two-letter Fibonacci word (see
\seqnum{A003849} for an extensive bibliography).
In this section we discuss various properties of
$\TTW$  for use later in the paper. Most of the properties  are analogs of similar
properties of the Fibonacci word.


\subsection{The Tribonacci representation of numbers}\label{Sec5_1}

Define the Tribonacci numbers $\{ T_n : n \geq 0\}$ as follows:
$T_{-3} = 0$, $T_{-2} = 0$,
$T_{-1} = 1$, and $T_n = T_{n-1} + T_{n-2} + T_{n-3}$
for $n \geq 0$ (\seqnum{A000073}).

For $E = e_1 e_2 \cdots e_i$
{\it any finite string\/} of $0$'s and $1$'s,
let $[E]_T$ be the
number $n = \sum_{1 \leq j \leq i} e_j T_{i-j}$; we call this
{\it a Tribonacci representation\/} for $n$.  Among all such representations,
one is {\it canonical\/}, that obtained by the greedy algorithm (repeatedly subtract
the largest possible Tribonacci number).  As
is well known (cf. \cite{CSH72}), the canonical representation is uniquely
characterized by not containing three consecutive $1$'s.
For integers $n \geq 0$, let $(n)_T$ be this canonical representation
for $n$, written with the least significant digit on the right.  For example,
$111$ and $1000$ are both representations of the
number $7$, but only the latter is canonical.
Note that we are distinguishing between $[E]_T$, which is a number,
and $(n)_T$, which is a binary string. 
The notations are combined in the formulas in Theorems~\ref{thm14}
and \ref{thm15}. 

\begin{lemma}\label{xylem}
Let $x, y$ be binary strings.   Then
$[x]_T = [y]_T$ if and only if $[x0]_T = [y0]_T$.
\end{lemma}

\begin{proof}
Let $z$ be the canonical Tribonacci representation of the integer
$[x]_T$.   Then $[z]_T = [x]_T$, so it suffices
to prove that $[x]_T = [z]_T$ if and only if $[x0]_T = [z0]_T$, where
$z$ is canonical.

Suppose $[x]_T = [z]_T$.  Consider obtaining $z$ by transforming
$x$ to remove the occurrence of three consecutive $1$'s, starting
with the most significant digit (at the left) and moving to the
least significant digit.  At each step we choose the leftmost
occurrence of $0111$ and replace it with $1000$.  (If $111$ appears
as a prefix, treat it as if it were $0111$.  We start at the left rather than the right,
for otherwise digits greater than $1$ could arise.)
Each replacement can create new copies of
$0111$ that also need to be changed, but these will occur
only to the left of the current position.
For example, if $x = 1011011101$, then the following transformations
take place (underlining highlights the block that is replaced):
$$ 1011\underline{0111}01 \rightarrow
1\underline{0111}00001 \rightarrow
1100000001 .$$
It is easy to see that this  normalization procedure
eventually halts, and
transforms any non-canonical binary
Tribonacci representation into a canonical one.

Now observe that if we carry out the same process
starting instead with $x0$, then the rightmost $0$
cannot participate in any of these replacements, and so we end
up with $z0$.  Hence $[x0]_T = [z0]_T$.

On the other hand, if $z$ is canonical, then so is $z0$.
So if $[x0]_T = [z0]_T$, we can obtain $z0$ by processing
the representation of $x0$ as above.   If we make the replacements
of $0111$ with $1000$ from left to right, as before, then the
last $0$ of both representations cannot participate in a
replacement, and so omitting the last $0$ gives exactly the same
sequence of replacements.  Hence $[x]_T = [z]_T$.
\end{proof}

Note that the hypothesis that $x$ and $y$ are binary strings is
necessary.   For example, if we write $4$ in a non-canonical way as $[20]_T = [100]_T$, then
$8 = [200]_T \not= [1000]_T = 7$.

\begin{corollary}\label{cor2}
Suppose $x, y, w_1, w_2$ are binary strings such that $|w_1| = |w_2|$,
$[x w_1]_T = [y w_2]_T$, and $[w_1]_T = [w_2]_T$.  Then
$[x]_T = [y]_T$.
\end{corollary}

\begin{proof}
Take the equality $[xw_1]_T = [yw_2]_T$ and subtract
the equality $[w_1]_T = [w_2]_T$ from it.
The result is $[x 0^i]_T = [y 0^i]_T$, where $i = |w_1| = |w_2|$.
Then by applying Lemma~\ref{xylem} $i$ times, we get $[x]_T = [y]_T$.
\end{proof}

\begin{corollary}\label{cor5}
Let $e_1 \cdots e_i$ be a binary Tribonacci representation for $n$.
Then the quantity
$[e_1 \cdots e_{i-1}]_T + e_i$ does not depend on the particular
representation $e_1 \cdots e_i$ chosen for $n$.
\end{corollary}

\begin{proof}
Let $E = e_1 \cdots e_i$ and $F = f_1 \cdots f_j$ be two binary Tribonacci representations for $n$.   Without loss of generality, we can assume that one of the
representations is canonical.

If $e_i = f_j$, then
by Corollary~\ref{cor2}
we get $[e_1 \cdots e_{i-1}]_T = [f_1 \cdots f_{j-1}]_T$,
and hence 
$$
[e_1 \cdots e_{i-1}]_T + e_i = [f_1 \cdots f_{j-1}]_T + f_j.
$$

Otherwise assume that $e_i = 0$ and $f_j = 1$.   If $F$ is
the canonical representation for $n$, then in carrying out the normalization
procedure to convert $E$ to $F$
(as we did in the proof
of Lemma~\ref{xylem}),
we evidently cannot change $E$'s last bit, so $f_j = 0$, a contradiction.
So $E$ must be the canonical representation for $n$.
Now consider carrying out the normalization procedure to convert
$F$ to $E$.
To change $f_j = 1$ into the $0$ corresponding to $e_i$,
the only possibility is that the
rightmost four bits of $F$ are $0111$ and the rightmost
four bits of $E$ are $1000$.
Write $E = E'1000$ and $F = F' 0111$.
Then $[E'1000]_T = [F'0111]_T$, and by Corollary~\ref{cor2}
we get $[E']_T = [F']_T$.  By applying Lemma~\ref{xylem} three times, we get
$[E'000]_T = [F'000]_T$.    Adding $[100]_T + 0 = [011]_T + 1$ to both sides
gives $[E'100] + 0 = [F'011]_T + 1$, as desired.
\end{proof}


\subsection{The Tribonacci morphism}\label{Sec5_2}

Using the  Tribonacci morphism~\eqn{Eqtau}, we
define a sequence of finite binary words by
${\bf T}_n := \tau^n (a)$ for $n \geq 0$.
Then the Tribonacci word is $\TTW =  \lim_{n \to \infty} {\bf T}_n$.

\begin{theorem}
We have
\begin{align*}
\tau^n (a) &= {\bf T}_n, \quad n \geq 0, \\
\tau^n (b) &= {\bf T}_{n-1} \cdot {\bf T}_{n-2}, \quad n \geq 1,\\
\tau^n (c) &= {\bf T}_{n-1}, \quad n \geq 0.
\end{align*}
\end{theorem}

\begin{proof}
An easy induction on $n$.
\end{proof}

\begin{lemma}\label{lem7}
We have ${\bf T}_n = {\bf T}_{n-1} \cdot {\bf T}_{n-2}
\cdot {\bf T}_{n-3}$ for $n \geq 3$.
\end{lemma}

\begin{proof}
We have 
\begin{align*}
{\bf T}_n &= \tau^n (a) = \tau^{n-3} (\tau^3(a)) =
\tau^{n-3} (abacaba) = \tau^{n-3} (abac) \tau^{n-3} (ab) \tau^{n-3}(a) \\
&= \tau^{n-1} (a) \tau^{n-2} (a)  \tau^{n-3} (a) =
{\bf T}_{n-1} \cdot {\bf T}_{n-2} \cdot {\bf T}_{n-3}.
\end{align*}
\end{proof}

The next two lemmas are also easily established by induction:

\begin{lemma}
We have $|{\bf T}_n| = T_n$ for $n \geq 0$.
\end{lemma}

\begin{lemma}\label{lem9}
For $n \geq 0$,  ${\bf T}_n$ contains
$T_{n-1}$ $a$'s,
$T_{n-2}$ $b$'s, and
$T_{n-3}$ $c$'s.
\end{lemma}


\subsection{The Tribonacci word}\label{Sec5_3}

\begin{lemma}\label{lem10}
Let $U_m$ be the set of length-$m$ binary strings consisting of the
Tribonacci representations (padded with leading zeros, if necessary)
of the numbers from $0$ to $T_m - 1$.  Then
$$ U_m = 0 U_{m-1} \sqcup 10 U_{m-2} \sqcup 110 U_{m-3},$$
for $m \geq 3$, where $\sqcup$ denotes disjoint union.
\end{lemma}

\begin{proof}
Again an easy induction on $m$.
\end{proof}

\begin{theorem}\label{thm11} {\rm(\cite{DuRi08})}.
Let the Tribonacci representation of $n-1$
be $e_1 e_2 \cdots e_i 0 1^j$, where $j \in \{ 0,1,2\}$.
Then
$$
t_n = \begin{cases}
a,& \text{if $j = 0$,} \\
b,& \text{if $j = 1$,} \\
c,& \text{if $j = 2$.}
\end{cases}
$$
\end{theorem}

\begin{proof}
An easy proof by induction on $n$, where $T_m +1 \le n \le T_{m+1}$,
using Lemma~\ref{lem9}.
\end{proof}

For $n \geq 1$ let $N_a(n), N_b(n), N_c(n)$ be the number of $a$'s, $b$'s, $c$'s respectively in $t_1 \cdots t_n$.

\begin{theorem}\label{thm12}
For $n \geq 1$ let the Tribonacci representation of $n$ be $e_1 e_2 \cdots e_i$.
Then
\begin{align*}
N_a(n) &= [e_1 \cdots e_{i-1}]_T + e_i, \\
N_b(n) &= [e_1 \cdots e_{i-2}]_T + e_{i-1}, \\
N_c(n) &= [e_1 \cdots e_{i-3}]_T + e_{i-2}.
\end{align*}
\end{theorem}

\begin{proof}
We know from Theorem~\ref{thm11} that $t_n$ depends on the Tribonacci 
representation of $n-1$ rather than $n$,
so for this proof we set $\nu = n-1$. The proof is by induction on $\nu$.
Consider a $\nu$ in the range   $T_m \leq \nu < T_{m+1}$.
We prove the result for $N_a(n)$, with the other results being
exactly analogous. The base cases are easy.

For the induction step, there are two cases to consider:
(a) $T_m \leq \nu < T_m + T_{m-1}$ and
(b) $T_m + T_{m-1} \leq  \nu  < T_{m+1}$.

If (a) holds, then write
$\nu = T_m + \nu'$ with $0 \le \nu' < T_{m-1}$, and consider
the length-$m$ Tribonacci representations of the numbers
from $0$ to $\nu-1$.
Using Lemma~\ref{lem10},
we see that the length-$m$
representations of the numbers from $0$ to $T_m - 1$
start with $0$, while the numbers
from $T_m$ to $T_m + \nu'-1$ have representations that
start with $10$.  For these latter numbers, subtracting
$T_m$ removes a leading $1$ from the Tribonacci representation,
and using Lemma~\ref{lem7} we see
that the Tribonacci representations
of the numbers from $T_m$ to $T_m + \nu' - 1$
are those of the numbers from $0$ to $\nu'-1$, except for
this leading $1$.
By Theorem~\ref{thm11} it follows that
the number of $a$'s in 
$t_{T_m+1}\cdots t_{T_m + \nu'}$ is equal to $N_a(\nu')$.

Hence
\begin{align*}
N_a(n) &= N_a(T_m) + N_a(\nu') \\
&= T_{m-1} +N_a(\nu') \quad \mbox{(by~Lemma ~\ref{lem9})} \\
&= T_{m-1} + [e_3 \cdots e_{i-1}]_T + e_i \quad \mbox{(by~the induction hypothesis)}  \\
&= [e_1 \cdots e_{i-1}]_T + e_i.
\end{align*}

Case (b) is similar.
\end{proof}

\begin{theorem}\label{thm13}
Let the Tribonacci representation of $n-1$ be
$e_1 \cdots e_i$, the representation of $n-2$ be
$f_1 \cdots f_j$, and the representation of
$n-4$ be $g_1 \cdots g_k$.  Then
\begin{align*}
N_a(n) &=  [e_1 \cdots e_{i-1}]_T + 1, \\
N_b(n) &=  [f_1 \cdots f_{j-2}]_T + 1, \\
N_c(n) &=  [g_1 \cdots g_{k-3}]_T + 1 .
\end{align*}
\end{theorem}

\begin{proof}
We prove the result for $N_a(n)$, the others being similar.
There are three cases, depending on whether $(n-1)_T = x0$,
$x01$, or $x011$.
If $(n-1)_T = x0$, then $x1$ is a representation of $n$;
if $(n-1)_T = x01$,   $x10$ is a representation of $n$; and 
if  $(n-1)_T = x011$,  $x100$  is a representation of $n$.
The result then follows by combining Corollary~\ref{cor5} with
Theorem~\ref{thm12}
\end{proof}

An analogous property to Theorem~\ref{thm13} for the Fibonacci word was established in \cite[Section~5]{DMSS16}.


\subsection{The indexing sequences 
\texorpdfstring{$A_n$}{A},
\texorpdfstring{$B_n$}{B},
\texorpdfstring{$C_n$}{C}.}\label{Sec5_4}

Let  $A_n$ ($n \ge 1$) denote the index of the $n$th
occurrence of the letter $a$ in $\TTW$, with similar
definitions for $B_n$ and $C_n$. We also set $A_0 = B_0 = C_0 = 0$.
The initial values of $A_n$, $B_n$, $C_n$ are
\beql{TabABC}
\begin{array}{|l|rrrrrrrrrrrrrr|}
\hline
n & 0 & 1 & 2 & 3 & 4 & 5 & 6 & 7 & 8 & 9 & 10 & 11 & 12 & \ldots \\
\hline
A_n & 0 & 1 & 3 & 5 & 7 & 8 & 10 & 12 & 14 & 16 & 18 & 20 & 21 & \ldots \\
B_n & 0 & 2 & 6 & 9 & 13 & 15 & 19 & 22 & 26 & 30 & 33 & 37 & 39 & \ldots \\
C_n & 0 & 4 & 11 & 17 & 24 & 28 & 35 & 41 & 48 & 55 & 61 & 68 & 72 & \dots \\
\hline
\end{array}
\eeq

\vspace{0.05in}
For further
terms see Table~\ref{TabXYMP2} or 
\seqnum{A003144}, \seqnum{A003145}, \seqnum{A003146}.
These sequences are inverses to
the sequences $N_a(n)$, $N_b(n)$, $N_c(n)$  defined in~\S\ref{Sec5_3}.
For example, $A_5 = 8$, while $N_a(8) = 5$.
They are studied in many references (\cite{BBB04, CSH72, DuRi08}).

\begin{theorem}\label{thm14}
We have
\begin{itemize}
\item $A_n =   [ (n-1)_T 0 ]_T + 1$,
\item $B_n =   [ (n-1)_T 01 ]_T + 1$,
\item $C_n =  [ (n-1)_T 011]_T + 1$.
\end{itemize}
\end{theorem}

\begin{proof}
We use Theorem~\ref{thm12}, which tells us how many $a$'s (resp., $b$'s, $c$'s)
occur in a prefix of $\TTW$ of
a given length.
 We prove the result for $A_n$, with the other results being proved
analogously.
Let us count how many $a$'s
there are in a prefix $[(n-1)_T 0]_T$.  By Theorem~\ref{thm12},
there are $n-1$ of them.   Similarly, Theorem~\ref{thm12} says
that there are $n$ $a$'s in the prefix of length $[(n-1)_T 1]_T$.
Since we index $\TTW$ starting at position $1$,
it now follows that the symbol at position $1+[(n-1)_T 0]$ must be an
$a$.
\end{proof}

We also record some further properties of $A_n$, $B_n$, $C_n$ established in \cite{DuRi08}.
For $n \ge 1$, we have
\begin{align}
A_n & ~=~   \mex\{ A_i, B_i, C_i : 0 \le i <n\},  \label{EqAmex} \\
B_n &~=~   A_n + \mex\{B_i-A_i, C_i-B_i : 0 \le i < n\},  \label{EqBmex} \\
C_n &~=~   A_n +B_n + n. \label{EqCmex}
\end{align}
There are some similarities with~\eqn{EqXYMP0}, and in particular~\eqn{EqBmex}
is consistent with equation~\eqn{EqXBA} of Corollary~\ref{CorXYABC} below, 
although we will not prove  this observation is correct until Section~\ref{SecMain}.
Furthermore, $\TTW$ is the unique ternary sequence satisfying 
\eqn{EqAmex}--\eqn{EqCmex}.
Also from \cite{DuRi08} (see Remarks 2.1--2.3), we know that 
\beql{EqDiffABC}
\{\Delta A\} = \THETA(2,2,1),~ \{\Delta B\} = \THETA(4,3,2),~ \{\Delta C\} = \THETA(7,6,4)\,.
\eeq

The next three properties are easy consequences of the definitions.
\begin{align} 
 A_n & ~=~   n+N_a(n-1)+N_b(n-1), \label{EqAED} \\
 B_n & ~=~   A_n + N_a(A_n)+N_b(A_n) ,  \label{EqBED} \\
 C_n & ~=~   B_n + N_a(B_n) + N_b(B_n). \label{EqCED} 
\end{align}
We prove \eqn{EqAED}, since the same argument will be used later.
Since $\TTW$ is the fixed point of $\tau$, $\TTW=\tau(\TTW)$.
Writing $\tau(t_n)$ underneath $t_n$, we see:
$$
\begin{array}{ccccccccccccccc}
n:   & 1 & 2 & 3 & 4 & 5 & 6 & 7 & 8 & 9 & 10 & 11 & 12 & 13 & \ldots \\
t_n: & a & b & a & c & a & b & a & a & b & a & c & a & b & \dots\ \\
\tau(t_n): & ab & ac & ab & a & ab & ac & ab & ab & ac & ab & a & ab & ac & \ldots \\
\end{array}
$$
Each letter $t_n$ in $\TTW$ produces an $a$ in $\tau(t_n)$, and this $a$ is at position $p$,
where $p$ equals $n$ plus the total number of $a$'s and $b$'s before $t_n$. This is exactly the assertion \eqn{EqAED}.
Properties \eqn{EqBED} and \eqn{EqCED} have similar proofs.

Many other properties are known, such as (\cite{BBB04,  DuRi08})
\beql{EqMOP}
A_{A_n}+1=B_n,~ A_{B_n}=B_{A_n}+1, ~A_{B_n}+1=C_n.
\eeq


\subsection{Numerical bounds.}\label{Sec5_5}

Appending a $0$ to the Tribonacci representation of a number (as 
in the formula for $A_n$ given  in Theorem~\ref{thm14}) 
has about the same effect as  multiplying the number by $\psi$.
To get precise estimates we must study the Tribonacci numbers $T_n$ themselves.

From the theory of linear recurrences we know that if $\psi$, $\psi_2$ and $\psi_3 := \overline{\psi_2}$ denote
the roots of $x^3=x^2+x+1$ then there are constants 
$c_1$, $c_2$, and $c_3 := \overline{c_2}$ such that 
\beql{EqBinet}
T_n ~=~ c_1 \,\psi^n + c_2\,\psi_2^n + c_3\, \psi_3^n, \quad \mbox{for~} n \ge 0.
\eeq
The numerical values of these constants are
$\psi_2 = -0.419643\ldots + 0.606291\ldots i$,
$c_1 = 0.336228\ldots$, $c_2 = -0.168114\ldots - 0.198324\ldots i$.

From \eqn{EqBinet} it is straightforward to show that, for $n \ge 0$, we have
\beql{Eqest1}
| T_n - c_1\,\psi ^n | \le 0.283 \, (0.738)^n
\eeq
and
\begin{align}
| T_{n+1} - \psi T_n | & \le 0.731 \, (0.738)^n,  \label{Eqest2} \\
| T_{n+2} - \psi^2 T_n | &  \le 1.113 \, (0.738)^n,  \\
| T_{n+3} - \psi^3 T_n | &  \le 1.877 \, (0.738)^n. \\
\notag
\end{align}

\begin{theorem}\label{thm15}
For all $n \geq 0$ we have
\begin{align}
 -0.596 & < [(n)_T 0]_T - \psi n  < 0.856, \label{Eqest4} \\
 -0.883 & < [(n)_T 00]_T - \psi^2 n  < 1.460, \\
 -1.461 & < [(n)_T 000]_T - \psi^3 n  < 2.298.
\end{align}
\end{theorem}

\begin{proof}
We prove \eqn{Eqest4}, the arguments in the other two cases being similar.
We write $n$ in its canonical Tribonacci representation,
say $n = T_{e_1} + T_{e_2} + \cdots + T_{e_s}$, where
$e_1 > e_2 > \cdots > e_s$.
Then $[(n)_T 0]_T = T_{e_1+1} + T_{e_2+1} + \cdots + T_{e_s + 1}$, so
$$
[(n)_T 0]_T - \psi n = \sum_{1 \leq j \leq s} (T_{e_j+1} - \psi T_{e_j}) .
$$
We break up this sum into two parts, one where
$2 \leq e_j \leq 18$, and one where $e_j > 18$.
From \eqn{Eqest2}, the latter sum is bounded in absolute value
by $\sum_{j \geq 19} 0.731 \cdot 0.738^{\,j} \leq .009$.
The former sum can be bounded by actually computing it for all $n < T_{19} =
121415$.  The minimum is achieved at
$n = 65915$ and is, rounded down, equal to
$-0.587$.  The maximum is achieved at
$n = 78748$ and is, rounded up, equal to $0.847$.
Hence $-0.596 < [(n)_T 0]_T - \alpha_1 n < 0.856$.
\end{proof}

\begin{theorem}\label{thm16}
For all $n \ge 0$,
\begin{align} 
\lfloor \psi n \rfloor -1 ~ \le ~ & A_n ~ \le ~ \lfloor \psi n \rfloor +1, \label{EqAbnds}   \\
\lfloor \psi^2 n \rfloor -2 ~ \le ~  & B_n ~ \le  ~  \lfloor \psi^2 n \rfloor +1, \label{EqBbnds}   \\
\lfloor \psi^3 n \rfloor -3 ~ \le ~ & C_n ~  \le  ~ \lfloor \psi^3 n \rfloor +1. \label{EqCbnds}   
\end{align}
\end{theorem}
\begin{proof}
From Theorem~\ref{thm15} we get
$-.596 < [(n-1)_T 0]_T - \psi (n-1)< .856$.
Since
$A_n =  [(n-1)_T 0 ]_T + 1$, we have
$-.596 < A_n - 1 - \psi (n-1) < .856$.
Hence
$-.017 < \psi n - A_n < 1.436$, so
$A_n -.017 < \psi n < A_n + 1.436$.
Taking floors gives 
$$
A_n-1 \leq \lfloor \psi n \rfloor \leq A_n + 1,
$$
which proves the first assertion. The other two have similar proofs, which we omit.
\end{proof}
 
 Values of $n$ for which the ``$+\,1$'' on the right-hand sides of \eqn{EqAbnds}--\eqn{EqCbnds}
 are actually needed seem quite rare,
 the first instances occurring at $n=12737$, $329$, and $2047$, respectively 
  (see \seqnum{A275158}, \seqnum{A278352}, \seqnum{A278353}).
 
For use in the proof of Theorem~\ref{ThMT1} we note that 
\beql{EqpsiC}
\psi C_n > C_{n+1} \mbox{~for~} n \ge 2\,.
\eeq

Bounds on the inverse quantities are easier to derive, and
we just state the result.
Deleting the least significant bit of
 the Tribonacci representation of $n$
has about the same effect as  dividing $n$ by $\psi$.
From Theorem~\ref{thm12} we  obtain:

\begin{theorem}\label{thm17}
For all $n \ge 1$,
\begin{align}
\left\lfloor \frac{n}{\psi} \right\rfloor \le \frac{n}{\psi} & \le N_a(n) \le \left\lfloor \frac{n}{\psi} \right\rfloor +1,  \label{EqNabnd} \\
\left\lfloor \frac{n}{\psi^2} \right\rfloor \le \frac{n}{\psi^2} & \le N_b(n) \le \left\lfloor \frac{n}{\psi^2} \right\rfloor +1,  \label{EqNbbnd} \\
\left\lfloor \frac{n}{\psi^3} \right\rfloor \le \frac{n}{\psi^3} & \le N_c(n) \le \left\lfloor \frac{n}{\psi^3} \right\rfloor +1. \label{EqNcbnd}
\end{align}
\end{theorem}


\section{The main theorem}\label{SecMain}
We can now establish our main theorem, Theorem~\ref{ThMT1}, which shows
that the rows of the $XYMP$ table are in one-to-one
correspondence with the terms of the theme song $\THETA$ (or, 
if we ignore the initial $n=0$ term,  with the terms of the Tribonacci word $\TTW$).
The bijection can be seen in the second column of Table~\ref{TabXYMP2}.

\begin{lemma}\label{Lemma1}
Let $\tau'$ be the morphism
$\tau':  a \to ba,  b \to ca,  c \to a$.
Then 
\beql{EqLemma1}
\tau'(c \cdot \TTW(a,b,c)) ~=~ \TTW(a,b,c)\,.
\eeq
\end{lemma}
\begin{proof}
Let  $A'_n$ ($n \ge 1$) denote the index of the $n$th
occurrence of the letter $a$ in the image $\tau'(c \cdot \TTW(a,b,c))$, with similar
definitions for $B'_n$ and $C'_n$.
Write the terms of $\tau'(c \cdot \TTW(a,b,c))$ under the corresponding terms of $c \cdot \TTW(a,b,c)$,
as we did in the proof of~\eqn{EqAED}--\eqn{EqCED} above.
As in that proof, we observe that each term of $c \cdot \TTW(a,b,c)$
produces an $a$ in the image, only now the initial $c$ produces an extra $a$ at the start, and 
the $a$ produced by $t_n \in \TTW(a,b,c)$
is displaced from $n+1$ by the number of copies of $a$ and $b$ in $\TTW(a,b,c)$
{\em at or before} $t_n$. This implies that
$$
A'_{n+1} = n+1 + N_a(n)+N_b(n).
$$
So $A'_n = n + N_a(n-1)+N_b(n-1) = A_n$. Similar arguments show that 
$B'_n=B_n$ and $C'_n=C_n$, and so \eqn{EqLemma1} holds.  
\end{proof}

\begin{lemma}\label{Lemma2}
Let $\tau''$ be the morphism
$\tau'':  a \to acab,  b \to aab,  c \to ab$.
Then 
\beql{EqLemma2}
\tau''(c \cdot \TTW(a,b,c)) ~=~ \TTW(a,b,c)\,.
\eeq
\end{lemma}
\begin{proof}
We have $\tau''= \tau \circ \tau'$ (first apply $\tau'$ then $\tau$),
and $\tau$ fixes $\TTW(a,b,c)$, so the result
follows from Lemma~\ref{Lemma1}.
\end{proof}

\begin{theorem}\label{ThMT1}
We have
\begin{align} \label{EqMT}
\{ \Delta X \}  &  ~=~   \THETA(2,1,1), \nonumber \\
\{ \Delta Y \}  &  ~=~   \THETA(3,3,2), \nonumber \\
\{ \Delta M \}  &  ~=~   \THETA(1,2,1), \nonumber \\
\{ \Delta P \}  &  ~=~   \THETA(5,4,3). \nonumber \\
\end{align}
\end{theorem}

\begin{proof}

There are three parts to the proof. 
In the first part we  show
that, for any $i \ge 1$, if  terms $0$ through $C_i-1$ of $\{\Delta Y\}$ 
(that is, the terms until just before the $i$th occurrence of $c=2$ in $\{\Delta Y\}$)
agree with  terms $0$ through  $C_i-1$ of $\THETA(3,3,2)$, then
terms $0$ through $\lfloor \psi (C_i-1) \rfloor $  of $\{\Delta X\}$ agree
with  terms $0$ through 
$\lfloor \psi (C_i-1) \rfloor$  of $\THETA(2,1,1)$.  

We form the partial sums of the $\{\Delta Y\}$ to get $C_i$ terms
of $\{ Y\}$, compute the complement to get a certain number of  initial terms
of $\{X\}$, and take the differences to get the initial terms of $\{\Delta X\}$.

A $3$ in the $\{\Delta Y\}$ sequence corresponds to a succession 
$\ldots, x, x+2, x+3, x+5, \ldots$
of terms in $\{X\}$, with differences $\ldots, 2,1,2,\ldots\,$.
The initial $2$ was already
present in $\{\Delta X\}$, so the $3$ in $\{\Delta Y\}$ produces a pair $1,2$ in $\{\Delta X\}$.
Similarly, a term $\Delta Y=2$ produces a single $2$ in $\{\Delta X\}$.
The initial $C_i$ terms of $\{\Delta Y\}$ are transformed by the action
of the map $3 \to 1,2$; $2 \to 2$ into a sequence over the
alphabet $\{1,2\}$. We can state this assertion
in an equivalent way. We identify the initial terms of 
$\{ \Delta Y\}$ with the initial terms of $\THETA(a,b,c) = \THETA(3,3,2)$,
and the initial terms of 
$\{ \Delta X\}$ with the initial terms of $\THETA(a,b,c) = \THETA(2,1,1)$.
The map can be described as $\tau': a \to ba, b \to ca, c\to a$,
and Lemma~\ref{Lemma1} guarantees that the image is indeed 
a prefix of $\THETA(2,1,1)$.

\begin{table}[htb]
\caption{ Showing how the difference sequence $\{\Delta Y\}$ 
generates a larger number of terms of the difference sequence $\{\Delta X\}$. The first four lines refer to $\{\Delta Y\}$ and $\{Y\}$, the second four to $\{X\}$ and $\{\Delta X\}$. }
\label{TabMT1}
\footnotesize
$$
\begin{array}{ccccccccccccccccccccc}
\hline \\
n:           & 0 & ~ & 1 & ~ & ~ & 2 & ~ & ~ & 3 & ~ & ~ & 4 & ~ & 5 & ~ & ~ & 6 & ~ & ~ & 7 \\
\THETA:      & c & ~ & a & ~ & ~ & b & ~ & ~ & a & ~ & ~ & c & ~ & a & ~ & ~ & b & ~ & ~ & a \\
\Delta Y_n:  & 2 & ~ & 3 & ~ & ~ & 3 & ~ & ~ & 3 & ~ & ~ & 2 & ~ & 3 & ~ & ~ & 3 & ~ & ~ & 3 \\
Y_n:         & 0 & ~ & 2 & ~ & ~ & 5 & ~ & ~ & 8 & ~ & ~ &11 & ~ &13 & ~ & ~ &16 & ~ & ~ &19 \\
X_n:         & 0 & 1 & ~ & 3 & 4 & ~ & 6 & 7 & ~ & 9 &10 & ~ &12 & ~ &14 &15 & ~ &17 &18 & ~ \\
\Delta X_n:  & 1 & ~ & 2 & 1 & ~ & 2 & 1 & ~ & 2 & 1 & ~ & 2 & ~ & 2 & 1 & ~ & 2 & 1 & ~ & 2 \\
\THETA:      & c & ~ & a & b & ~ & a & c & ~ & a & b & ~ & a & ~ & a & b & ~ & a & c & ~ & a \\
n:           & 0 & ~ & 1 & 2 & ~ & 3 & 4 & ~ & 5 & 6 & ~ & 7 & ~ & 8 & 9 & ~ &10 &11 & ~ &12 \\
\hline \\
\hline \\ 
n:           & ~ & ~ & 8 & ~ & ~ & 9 & ~ & ~ &10 & ~ & ~ &11 & ~ &12 & ~ & ~ &13 & ~ & ~ &14 \\
\THETA:      & ~ & ~ & a & ~ & ~ & b & ~ & ~ & a & ~ & ~ & c & ~ & a & ~ & ~ & b & ~ & ~ & a \\
\Delta Y_n:  & ~ & ~ & 3 & ~ & ~ & 3 & ~ & ~ & 3 & ~ & ~ & 2 & ~ & 3 & ~ & ~ & 3 & ~ & ~ & 3 \\
Y_n:         & ~ & ~ &22 & ~ & ~ &25 & ~ & ~ &28 & ~ & ~ &31 & ~ &33 & ~ & ~ &36 & ~ & ~ &39 \\
X_n:         &20 &21 & ~ &23 &24 & ~ &26 &27 & ~ &29 &30 & ~ &32 & ~ &34 &35 & ~ &37 &38 & ~ \\
\Delta X_n:  & 1 & ~ & 2 & 1 & ~ & 2 & 1 & ~ & 2 & 1 & ~ & 2 & ~ & 2 & 1 & ~ & 2 & 1 & ~ & 2 \\
\THETA:      & b & ~ & a & b & ~ & a & c & ~ & a & b & ~ & a & ~ & a & b & ~ & a & c & ~ & a \\
n:           &13 & ~ &14 &15 & ~ &16 &17 & ~ &18 &19 & ~ &20 & ~ &21 &22 & ~ &23 &24 & ~ &25 \\
\hline \\
\end{array}
$$
\end{table}
\normalsize

This process is illustrated in Table~\ref{TabMT1}.
The first four rows show $\{\Delta Y\}$ as $\THETA(a,b,c) = \THETA(3,3,2)$, and 
its partial sums $\{Y\}$.  The second four rows show $\{X\}$ being
formed as the complement of $\{Y\}$, and the identification of
$\{\Delta X\}$ with  $\THETA(a,b,c) = \THETA(2,1,1)$.
 
From Theorem~\ref{Lem3}, in terms $1$ through $k := C_i-1$ of $\THETA(a,b,c)$ there are 
$N_a(k)$, $N_b(k)$, $N_c(k)$  copies of $a$, $b$, $c$, respectively.
After applying the map $\tau'$,  and taking into account
the slightly irregular behavior at the start of these
sequences, 
we obtain $2 + 2N_a(k)+2N_b(k)+N_c(k) = 2 + 2k-N_c(k)$ terms of $\{\Delta X\}$.
For illustration, in Table~\ref{TabMT1}, we may take $i=2$, $C_2=11$,
$k=10$, and using $N_a(10)=6$, $N_b(10)=3$, $N_c(10)=1$,
we see that we obtain $19+2 = 21$ terms of  $\{\Delta X\}$, that is, terms $0$ through $20$.

Using \eqn{EqNcbnd},
the number of terms we obtain is at least
$1 + k(2-1/\psi^3) = \psi k + 1 = \psi (C_i-1)+1$.
This completes the first part of the proof of the theorem.

In the second part of the proof we show
that, for any $i \ge 1$, if  terms $0$ through $C_i-1$ of $\{\Delta P\}$ 
agree with  terms $0$ through  $C_i-1$ of $\THETA(5,4,3)$, then
terms $0$ through $\lfloor \psi^2 (C_i-1) \rfloor$  of $\{\Delta M\}$ agree
with  terms $0$ through 
$\lfloor \psi^2 (C_i-1) \rfloor$  of $\THETA(1,2,1)$.  

The argument is parallel to that for the first part: we proceed from $\{\Delta P\}$
to $\{P\}$ to its complement, $\{M\}$, and then to $\{\Delta M\}$
(see Table~\ref{TabMT2}).  Every $5$ in $\{\Delta P\}$ produces 
a subsequence $1,1,1,2$ in $\{\Delta M\}$, every $4$ produces $1,1,2$,
and every $3$ produces $1,2$.

\begin{table}[htb]
\caption{ Showing how the difference sequence $\{\Delta P\}$ 
generates a larger number of terms of the difference sequence $\{\Delta M\}$. The first four lines refer to $\{\Delta P\}$ and $\{P\}$, the second four to $\{M\}$ and $\{\Delta M\}$. }
\label{TabMT2}
\footnotesize
$$
\begin{array}{ccccccccccccccccccccc}
\hline \\
n:           & 0 & ~ & ~ & 1 & ~ & ~ & ~ & ~ & 2 & ~ & ~ & ~ & 3 & ~ & ~ & ~ & ~ & 4 & ~ & ~ \\
\THETA:      & c & ~ & ~ & a & ~ & ~ & ~ & ~ & b & ~ & ~ & ~ & a & ~ & ~ & ~ & ~ & c & ~ & ~ \\
\Delta P_n:  & 3 & ~ & ~ & 5 & ~ & ~ & ~ & ~ & 4 & ~ & ~ & ~ & 5 & ~ & ~ & ~ & ~ & 3 & ~ & ~ \\
P_n:         & 0 & ~ & ~ & 3 & ~ & ~ & ~ & ~ & 8 & ~ & ~ & ~ &12 & ~ & ~ & ~ &   &17 & ~ & ~ \\
M_n:         & 0 & 1 & 2 & ~ & 4 & 5 & 6 & 7 & ~ & 9 &10 &11 & ~ &13 &14 &15 &16 & ~ &18 &19 \\
\Delta M_n:  & 1 & 1 & ~ & 2 & 1 & 1 & 1 & ~ & 2 & 1 & 1 & ~ & 2 & 1 & 1 & 1 & ~ & 2 & 1 & 1 \\
\THETA:      & c & a & ~ & b & a & c & a & ~ & b & a & a & ~ & b & a & c & a & ~ & b & a & ~ \\
n:           & 0 & 1 & ~ & 2 & 3 & 4 & 5 & ~ & 6 & 7 & 8 & ~ & 9 &10 &11 &12 & ~ &13 &14 & ~ \\
\hline \\
\hline \\
n:           & 5 & ~ & ~ & ~ & ~ & 6 & ~ & ~ & ~ & 7 & ~ & ~ & ~ & ~ & 8 & ~ & ~ & ~ & ~ & 9 \\
\THETA:      & a & ~ & ~ & ~ & ~ & b & ~ & ~ & ~ & a & ~ & ~ & ~ & ~ & a & ~ & ~ & ~ & ~ & b \\
\Delta P_n:  & 5 & ~ & ~ & ~ & ~ & 4 & ~ & ~ & ~ & 5 & ~ & ~ & ~ & ~ & 5 & ~ & ~ & ~ & ~ & 4 \\
P_n:         &20 & ~ & ~ & ~ & ~ &25 & ~ & ~ & ~ &29 & ~ & ~ & ~ & ~ &34 & ~ & ~ & ~ & ~ &39 \\
M_n:         & ~ &21 &22 &23 &24 & ~ &26 &27 &28 & ~ &30 &31 &32 &33 & ~ &35 &36 &37 &38 & ~ \\
\Delta M_n:  & 2 & 1 & 1 & 1 & ~ & 2 & 1 & 1 & ~ & 2 & 1 & 1 & 1 & ~ & 2 & 1 & 1 & 1 & ~ & 2 \\
\THETA:      & b & a & c & a & ~ & b & a & a & ~ & b & a & c & a & ~ & b & a & c & a & ~ & b \\
n:           &15 &16 &17 &18 & ~ &19 &20 &21 & ~ &22 &23 &24 &25 & ~ &26 &27 &28 &29 & ~ &30 \\
\hline \\
\end{array}
$$
\end{table}

\normalsize

We identify the initial terms of 
$\{ \Delta P\}$ with the initial terms of $\THETA(a,b,c) = \THETA(5,4,3)$,
and the initial terms of 
$\{ \Delta M\}$ with the initial terms of $\THETA(a,b,c) = \THETA(1,2,1)$.
The map can be described as $\tau'': a \to acab, b \to aab, c \to ab$,
and Lemma~\ref{Lemma2} guarantees that the image is indeed 
a prefix of $\THETA(1,2,1)$.

When $\tau''$ is applied to terms $0$ through $k := C_i-1$ of $\THETA(a,b,c)$  
we obtain $3 + 4N_a(k)+3N_b(k)+2N_c(k)$ terms of $\{\Delta M\}$.
From \eqn{EqNabnd}--\eqn{EqNcbnd}, this quantity is at least 
$3 + k(4/\psi+3/\psi^2+2/\psi^3) = \psi^2 k+3$, which is enough
to complete the second part of the proof. 

 For the third part of the proof we use induction. The induction hypothesis 
 is that for some $i \ge 2$,
 terms $0$ through $C_i-1$ of $\{\Delta Y\}$ 
agree with terms $0$ through  $C_i-1$ of $\THETA(3,3,2)$,
and  terms $0$ through $C_i-1$ of $\{\Delta P\}$ 
agree with terms $0$ through  $C_i-1$ of $\THETA(5,4,3)$.
From Table~\ref{TabXYMP2} we can verify that these assertions
 are true for all $i \le 3$. 
 
 From the first two parts of the proof the induction hypothesis implies
 that $\{\Delta X\}$ agrees with $\THETA(2,1,1)$ for
 $\psi C(i)$ terms, and  $\{\Delta M\}$ agrees with $\THETA(1,2,1)$
 for at least the same number of terms.
 Since $\{Y\} = \{X\}+\{M\}$,
 $\{\Delta Y\}$ agrees with $\THETA(2,1,1)+\THETA(1,2,1) = \THETA(3,3,2)$
 for $\psi C_i$ terms. Since $\{P\} = \{X\}+\{Y\}$,
 $\{\Delta P\}$ agrees with $\THETA(2,1,1)+\THETA(3,3,2) = \THETA(5,4,3)$
 for $\psi C_i$ terms. But by \eqn{EqpsiC},  $\psi C_i > C_{i+1}$ 
 (this is where we need $i \ge 2$), so  the induction hypothesis holds for $i+1$. 
 This completes the proof.
 \end{proof}
  
 \vspace*{+.3in}
\noindent{\bf Remarks.}

\noindent (i) The identification of $\{\Delta P\}$ with $\THETA(5,4,3)$ establishes the bijection
between the rows of the XYMP table and the terms
of the sequence $\THETA(a,b,c)$ (see column $2$ of Table~\ref{TabXYMP2}). The label for row $n \ge 0$ is
determined by the value of 
the quadruple $(\Delta X_n, \Delta Y_n, \Delta M_n, \Delta P_n)$:
if this is
\begin{align} \label{EqBij}
(2,3,1,5) & \text{~then~row~} n \text{~is labeled~} a, \nonumber \\
(1,3,2,4) & \text{~then~row~} n \text{~is labeled~} b, \nonumber \\
(1,2,1,3) & \text{~then~row~} n \text{~is labeled~} c. \nonumber \\
\end{align}
\noindent  (ii) Since $\Delta X_n = 2$ only in the first case, and
$\Delta M_n = 2$ only in the second case, we conclude that:

~~row $n$ is labeled $a$ if and only if $X_n+1 \in \{Y\}$,

~~row $n$ is labeled $b$ if and only if $M_n+1 \in \{P\}$, and 

~~otherwise row $n$ is labeled $c$.

\noindent (iii) We can now connect the $X_n$, $Y_n$, $M_n$, $P_n$ sequences
with the $A_n$, $B_n$, $C_n$ sequences.

\begin{corollary}\label{CorXYABC}
For $n \ge 0$,
\begin{align} 
X_n & = B_n-A_n, \label{EqXBA} \\
Y_n & = C_n-B_n, \label{EqYCB} \\
M_n & = C_n -2B_n+A_n, \label{EqMCBA} \\
P_n & = C_n-A_n. \label{EqPCA}
\end{align}\end{corollary}
\begin{proof}
These formulas follow from  \eqn{EqDiffABC}.
The first one, for example, follows because
$\{\Delta B\}  - \{\Delta A\}  = \THETA(4,3,2) - \THETA(2,2,1) = 
\THETA(2,1,1) = \{\Delta X\}$, by \eqn{EqMT},
and so $X_n  = B_n-A_n$.
\end{proof}
\noindent These formulas can be confirmed by looking
at the columns of Table~\ref{TabXYMP2}.

\noindent  (iv) The queens in the first quadrant (cf. Fig.~\ref{Fig3})
have coordinates $(X_n, Y_n)$, $n \ge 0$, and we can now
determine the slope of the line that they (approximately) lie on. For
\beql{EqSlope}
\frac{Y_n}{X_n} = \frac{C_n-B_n}{B_n-A_n},
\eeq
which from Theorem~\ref{thm16} converges to $(\psi^3-\psi^2)/(\psi^2-\psi) = \psi$
as $n$ increases.

\noindent  (v) We can also answer a question left over from Section~\ref{SecQOSS}.
From Theorem~\ref{ThMT1}, we see that $\Delta Y_n$ is never $4$, and
$\Delta P_n$ is never $6$. So there is no run of three consecutive
$X_n$ all differing by $1$, and no run of five consecutive
$M_n$ all differing by $1$.

\begin{figure}[!ht]
\centerline{\includegraphics[angle=0, width=7in]{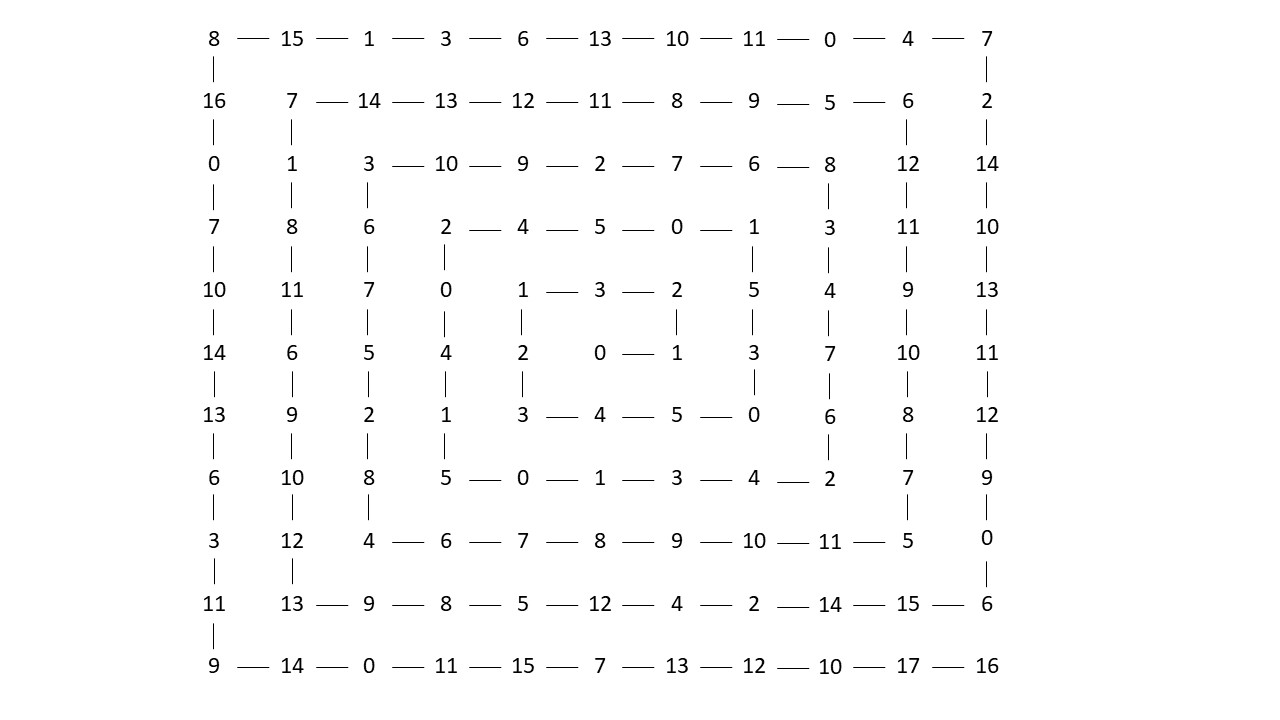}}
\caption{Sprague-Grundy values  (\seqnum{A274641})  for game based on
the queens-in-exile problem on a square spiral.
The cells with value $0$ are both the $\sP$-positions for the game and the locations of 
the exiled queens.  [Figure courtesy of Jessica Gonzalez.]}
\label{Fig4}
\end{figure}


\section{The Sprague-Grundy values}\label{SecSG}

The queens-in-exile problems  can be described
in terms of  two-person, impartial, combinatorial games.
We fix a numbering of the cells (such as along a square spiral,
or by antidiagonals). A queen is placed anywhere on the board, and the
players take turns, moving the queen to any lower-numbered cell which
is a queen's move away.
The first player who is unable to move loses.
The Sprague-Grundy value of a cell is the ``$\mex$'' of the values of the cells that are a
legal move away \cite{WW, Guy91}.
The $\sP$-positions are the cells with Sprague-Grundy value $0$, and are also
the locations of the exiled queens.

For the $\ZZ \times \ZZ$ board numbered along a square spiral, 
the Sprague-Grundy values are shown in 
Fig.~\ref{Fig4}.
By construction, the Sprague-Grundy values along any row are distinct,
and similarly for any column,
and any diagonal of slope $\pm 1$.   The following conjecture is very plausible but unsolved.

\begin{conjecture}\label{Conj1}
Every row, column, and diagonal of
slope $\pm 1$ in the array of Sprague-Grundy values
is a permutation of $\NN$.
\end{conjecture}

We do not even know, for example, that the numbers along the horizontal axis,
\beql{EqHztlaxis}
\ldots,24,12,16,9,14,6,5,4,2,0, 1,3,7,10,11,15,8,18,23,21,\ldots\, ,
\eeq
(\seqnum{A324778}, \seqnum{A324774}) include every nonnegative number.

\begin{table}[htb]
\caption{ Single-quadrant board numbered along 
upwards antidiagonals; circles indicate positions of exiled queens.}
\label{FigQ1a}
$$
\begin{array}{cccccccccc}
\circled{0}  &  2 & 5 & 9 & 14 & 20 & 27 & 35 & 44 &  \ldots \\
1  &  4 & 8 & \circled{13} & 19 & 26 & 34  & 43 & \ldots & \\
3  &  \circled{7} & 12 & 18 & 25 & 33 & 42 & \ldots & & \\
6  & 11 & 17 & 24 & \circled{32} & 41 & \ldots & & & \\ 
10 & 16 & \circled{23} & 31 & 40 & \ldots & & & & \\
15 & 22 & 30 & 39 &  \ldots & & & & & \\
21 & 29 & 38 & \ldots & & & & & & \\
28 &  37 & \ldots & & & & & & &   \\
36  & \ldots & & & & & & & &    \\
 \ldots & & & & & & &  & &  \\
\end{array}
$$
\end{table}


\section{The single-quadrant board}\label{SecQ1}
We have fewer results about the positions of the queens in this version of the problem,
so we will start right away with the combinatorial game.
This is played on an infinite 
 $\NN \times \NN$ board  where the squares are numbered along successive upward
antidiagonals, as shown in Table~\ref{FigQ1a}.
In the game, a queen is placed anywhere on the board, and the players take turns moving
it to a lower-numbered square that is a queen's-move away.
As usual the first player unable to move loses.
A small portion of the table of Sprague-Grundy values is shown in Table~\ref{FigQ1b},
and a color-coded illustration of the top $500 \times 500$
corner of the table  is given  in Fig.~\ref{FigColor}.

\begin{table}[htb]
\caption{ Sprague-Grundy values for single-quadrant board numbered along
upwards antidiagonals (\seqnum{A274528}). The indexing of the rows and columns in this table start at $0$.}
\label{FigQ1b}
$$
\begin{array}{cccccccccc}
0  &  2 & 1 & 5 & 3 & 4 & 9 & 10 & 12 & \ldots \\
1  &  3 & 4 & 0 & 7 & 2 & 5 & 11 &  \ldots & \\
2  &  0 & 5 & 1 & 8 & 6 & 4 & \ldots & & \\
3  &  1 & 2 & 4 & 0 & 7 & \ldots & & & \\ 
4 & 6 & 0 & 3 & 1 & \ldots & & & & \\
5 & 7 & 8 &  6 & \ldots & & & & & \\
6 & 4 & 3 & \ldots & & & & & & \\
7 & 5 & \ldots & & & & & & &   \\
8 & \ldots & & & & & & & &   \\
 \ldots & & & & & & &  &  & \\
\end{array}
$$
\end{table}

\begin{figure}[!ht]
\centerline{\includegraphics[angle=0, width=5in]{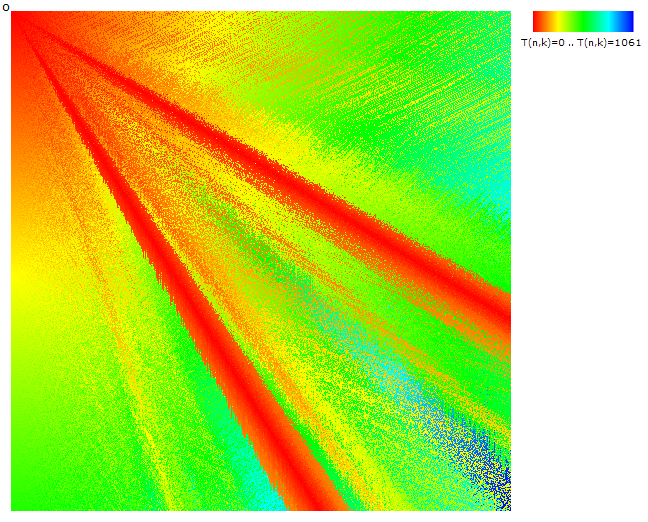}}
\caption{The top $500 \times 500$ corner
of the table  of Sprague-Grundy values. The color ranges from red
(for values near $0$ to blue
(for values near $1000$).  [Figure courtesy of Remy Sigrist.]}
\label{FigColor}
\end{figure}

The squares where the exiled queens settle in this version of
the problem (i.e.,  the $0$'s in the table)  have indices, reading along the successive antidiagonals,
\beql{EqA275897}
0, 7, 13, 23, 32, 96, 114, 142, 163, 183, 197, 261, 290, 446, 484, 581, \ldots
\eeq
(\seqnum{A275897}).
 The first five of these values are indicated by the circles
in Fig.~\ref{FigQ1a}.

An alternative way to specify the positions of the exiled queens is by the sequence $\{S_c: c \ge 0\}$
(\seqnum{A275895}),
which indicates  which row contains the queen in column $c$
(this is well-defined, thanks to Theorem~\ref{ThQ1} below).
The initial values of $S_c$ are
\beql{EqGQ}
0, 2, 4, 1, 3, 8, 10, 12, 14, 5, 7, 18, 6, 21, 9, 24, 26, 28, 30, 11, 13, 34, \ldots\,.
\eeq

The only theorem we have for this version of thev problem is:
\begin{theorem}\label{ThQ1}
Every column and every row of the table of Sprague-Grundy values
is a permutation of $\NN$.
\end{theorem}
\begin{proof} (Based on arguments in  \seqnum{A269526} given by Rob Pratt, Bob Selcoe, 
and N.J.A.S. in June 2016.)
There can be no repeated terms in any column, row, or diagonal, by construction,
so we must just show that there are no missing terms.
Consider column $c \ge 0$. 
Since the Sprague-Grundy values are calculated moving upwards along the antidiagonals, 
a number $k$ will appear in column $c$ unless it is blocked
by the presence of a $k$ in an earlier column.  But the first $c$ columns contain
at most $c$ copies of $k$, so eventually every $k$ will appear in column $c$.

Consider row $r \ge 0$, and suppose a number $k$ never appears.
There are at most $r$ copies of $k$ in the earlier rows, and
these can affect only a bounded portion of row $r$.
Consider a square $(r,n)$, where $n \ge 0$ is large. If $k$ is not to appear in that cell, there must be
a copy of $k$ in the antidiagonal to the South-West.
So in the right  triangle bounded by row $r$, column $0$, and the
antidiagonal through $(r,n)$, there must be at least $n+1-r$ copies of $k$ (allowing
for the $\le r$ copies of $k$ in the first $r$ rows).
Imagine these $k$'s replaced by chess queens. By construction they are
mutually non-attacking. But it is known (\cite[Problem~$252$]{VGL02},
or \seqnum{A274616})
that on a right triangular  half-chessboard of side $n$,
there can be at most $2n/3+1$ mutually non-attacking queens.
Since $2n/3+1 <n+1-r$ for large $n$, a $k$ must eventually appear in that row.
\end{proof}

As to the diagonals, although they {\em appear} to be permutations, there is no proof.
\begin{conjecture}\label{ConjDiag}
Every diagonal of the table of Sprague-Grundy values
is a permutation of $\NN$.
\end{conjecture}
\noindent The argument using non-attacking queens breaks down here because
the diagonal of the half-chessboard contains only half as many squares as the sides.
(The antidiagonals are certainly not permutations of $\NN$, since they have finite length.)

\begin{figure}[!ht]
\centerline{\includegraphics[angle=0, width=4in]{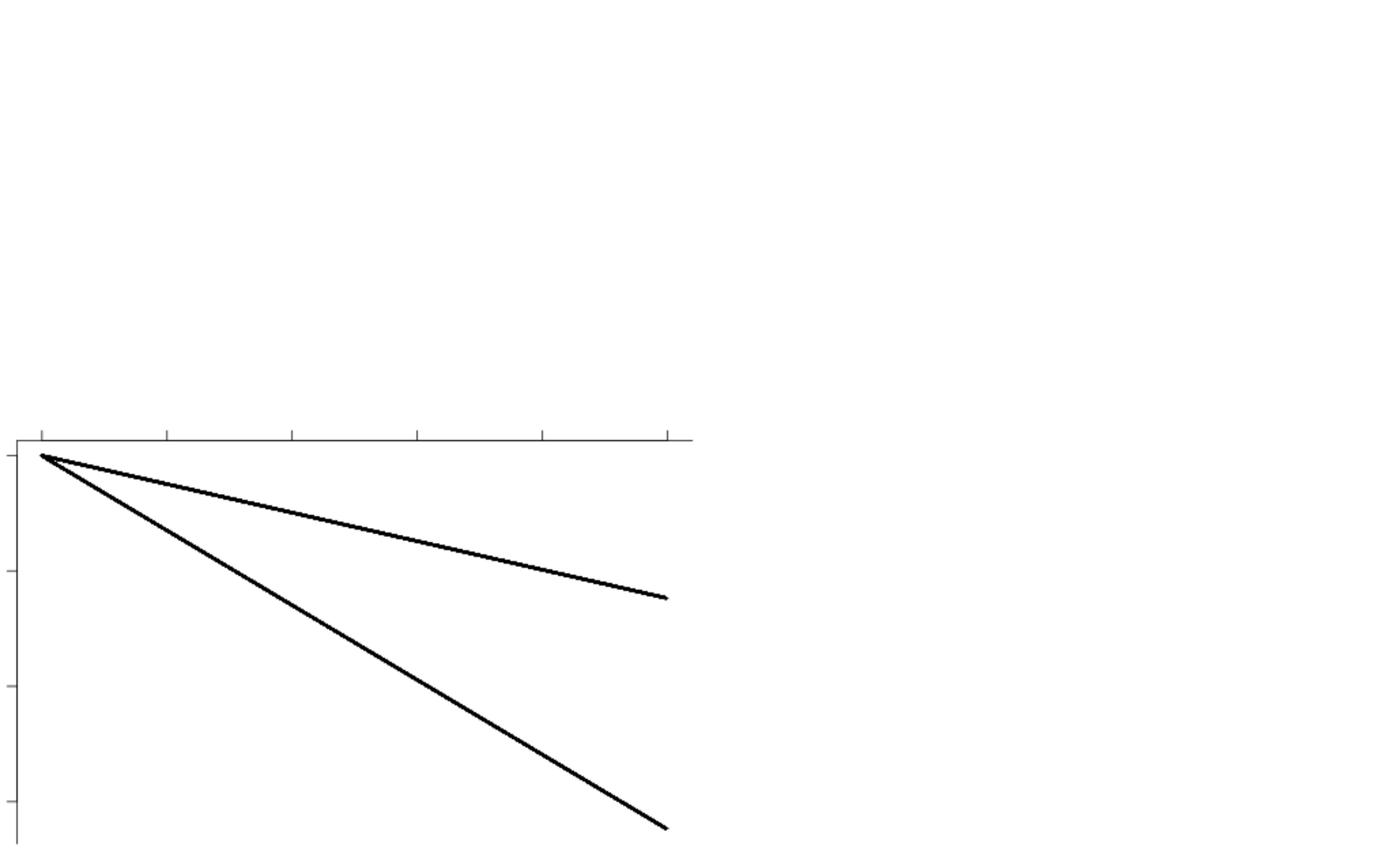}}
\caption{Positions of the first 10000 queens on the single-quadrant board.
The points appear to lie essentially on two straight lines.
The tick-marks on the horizontal axis (the column indices) are at $0$, $2000$, $4000, \ldots\,$, $10000$,
and on the vertical axis (the row indices)   at $0$, $5000$, $10000$, $15000$.}
\label{FigQ1}
\end{figure}

We return to the discussion of the positions of the exiled queens.
We observed in 2016 that the first $50000$ queens appear to lie 
almost exactly on two straight lines, of slopes $\phi$ and $1/\phi$, where $\phi$ is the golden ratio.
Figure~\ref{FigQ1} shows a plot of the first $10000$ queens. 
This is similar to what we saw in Fig.~\ref{Fig3},
only now we do not have a proof that the points lie on these lines, nor
do we have a proof that the slopes are what they appear to be.

Donald Knuth  (see Exercise 7.2.2.1--38  in \cite{DEKVol4B})
investigated this question in 2018, although his version of the problem is superficially
different. In his problem, the queens are  placed column-by-column.
A queen is placed in column $c \ge 0$ in the first square from which it cannot attack
any earlier queen.  However, it is easy to
prove that it makes no difference whether the board is scanned by successive antidiagonals
or by successive columns.
The same sequence $\{S_c\}$ is obtained in both cases.

Knuth has written an efficient program for computing this sequence,
and finds that the points are extremely close to the two lines: for $c  < 10^9$,
he finds that
\begin{align}\label{EqQ1bnda}
-2  < S_c -c\phi < 1, & \text{~~~if~}  S_c>c,  \\
-3  < S_c -\frac{c}{\phi} < 5, & \text{~~~if~}  S_c<c. 
\end{align}

The evidence for the next conjecture is therefore overwhelming:
\begin{conjecture}\label{ConjQ1}
There are constants $\epsilon_1$, $\epsilon_2$ such that
\begin{align} \label{EqQ1bndb}
|\,S_c -c\phi \,| < \epsilon_1, & \text{~~~if~}  S_c>c,  \\
\left| ~ S_c -\frac{c}{\phi}\right| < \epsilon_2, & \text{~~~if~}  S_c<c.  
\end{align}
\end{conjecture}

This combinatorial game can be described in a different way, which suggests 
a possible attack on the conjecture. If the queen being moved is
located at square $(x,y)$, we represent  its position by two piles of tokens, of sizes $x$ and $y$.
The legal moves are to remove any positive number of
tokens from one pile, or to remove an equal positive number of tokens from
both piles, or to move a positive number of tokens from the $y$ pile to the $x$ pile
(the latter move corresponds to moving the queen down
the antidiagonal). The first person unable to move loses.

A simpler two-pile game is the classic Wythoff's Nim (also called Wyt Queens) \cite{WW, Wyt07},
which has the same moves  except that moving tokens from one pile to another is
not allowed. 
The array of Sprague-Grundy values  can be
found in  \cite[First edition, Chap.~3, Table~3]{WW},  \cite{DFP99},  and  \seqnum{A004481}.
The sequence  $\{W_c\}$ specifying which rows the queens are in 
(the analog of $\{S_c\}$) is \seqnum{A002251}:
\beql{EqWN1}
0, 2, 1, 5, 7, 3, 10, 4, 13, 15, 6, 18, 20, 8, 23, 9, 26, 28, 11, 31, 12, 34, \dots\,.
\eeq
 Wythoff \cite{Wyt07} himself showed in 1907  that these points lie on two lines
(\seqnum{A000201}, \seqnum{A001950}):
\begin{align}\label{EqWN2}
W_c  = \lfloor c \phi \rfloor, & \text{~~~if~}  W_c>c,  \\
W_c  = \left\lfloor \frac{c}{\phi}\right\rfloor,  & \text{~~~if~}  W_c<c. 
\end{align}
The reason we mention this is that  Larsson and W\"{a}stlund \cite{LaWa14}
were able to analyze  the two lines of queens in Maharaja Nim, a variant of Wythoff's Nim.
  Can their method be adapted to our problem?

Although the array of Sprague-Grundy values for our problem looks irregular, it
appears that the columns eventually become quasi-periodic. Column~$1$ 
is $2,3,0,1,6,7,4, \ldots$ (\seqnum{A004482}), which is the Nim-sum $r \oplus 2$, and has generating function
\beql{EqGF1}
\frac{2-x-2x^2+3x^3}{(1-x)^2(1+x^2)} ~=~ \frac{(1+x)(2-x-2x^2+3x^3)}{(1-x)(1-x^4)} .
\eeq
All subsequent columns appear  to have a generating function with 
denominator $(1-x)(1-x^{16})$. Column $2$, for example, appears to have 
generating function $g(x)/((1-x)(1-x^{16}))$, where $g(x)$ is

\begin{align}
1+3\,x\,+&{x}^{2}-3\,{x}^{3}-2\,{x}^{4}+8\,{x}^{5}-5\,{x}^{6}+3\,{x}^{7} +{x}^{8}+5\,{x}^{9}+ \nonumber \\
+&{x}^{10}-3\,{x}^{11}+{x}^{12}-2\,{x}^{13}+8\,{x}^{14}-3\,{x}^{15}+2\,{x}^{17}. \nonumber
\end{align}


\begin{conjecture}\label{ConjGF}
Every column of the array of Sprague-Grundy values for the
single-quadrant problem has a 
a generating function with 
denominator $(1-x)(1-x^{16})$.
\end{conjecture}
\noindent (The conjecture also holds for columns $0$ and $1$.)
If true, this would mean that every column eventually becomes quasi-periodic with period $16$,
something that we find surprising.
No such property seems to hold for the rows of the table.
The quasi-periodicity of the Sprague-Grundy values for Wythoff's Nim and 
certain other combinatorial games was studied by
Dress, Flammenkamp, and Pink in \cite{DFP99}.  However, it does
not seem that Conjecture~\ref{ConjGF} follows from their work.
 

\section*{Acknowledgments}
Alois Heinz and R\'{e}my Sigrist   were always ready to design efficient
computer programs for these sequences, to generate extensive tables, and
to provide spectacular illustrations (such as Fig.~\ref{FigColor}).
Many of their  programs and tables can be found in
the entries in \cite{OEIS}.

We  thank Alois Heinz  for providing Fig.~\ref{Fig3}, and  Jessica Gonzalez
for drawing Figures~\ref{Fig2} and \ref{Fig4}.
We also  thank Donald E. Knuth for telling us about his investigations of the single-quadrant problem.
Achim Flannenkamp kindly provided a copy of \cite{DFP99}.


\medskip


\begin{thebibliography}{10}

\bibitem{AlSh03} 
J.-P. Allouche and J. Shallit,
\newblock {\em Automatic sequences}.
\newblock Cambridge University Press, 2003.

\bibitem{BBB04} 
E. Barcucci, L. Belanger, and S. Brlek, 
\newblock On Tribonacci sequences.
\newblock {\em Fibonacci Quarterly}, 42.4:314--320, 2004.

\bibitem{WW} 
E.~R.~Berlekamp, J.~H.~Conway, and R.~K.~Guy,
\newblock \emph{Winning ways for your mathematical plays}, 2nd ed.,  4~vols. 
\newblock A.~K.~Peters, Boston, 2004.

\bibitem{CSH72} 
L. Carlitz, R. Scoville, and V. E. Hoggatt, Jr., 
\newblock Fibonacci representations of higher order.
\newblock {\em Fibonacci Quarterly}, 10.1:43--69, 1972.

\bibitem{DFP99} 
A. Dress, A. Flammenkamp, and N. Pink, 
\newblock  Additive periodicity of the Sprague-Grundy function of certain Nim games.
\newblock  {\em Adv. Appl. Math.}, 22:249--270, 1999.

\bibitem{DMSS16} 
C. F. Du, H. Mousavi, L. Schaeffer, and J. Shallit, 
\newblock Decision algorithms for Fibonacci-automatic words, III: Enumeration and abelian properties.
\newblock {\em International Journal of Foundations of Computer Science},  27.08:943--963, 2016.

\bibitem{DFG17} 
E. Duch\^{e}ne,  A. S. Fraenkel, V. Gurvich, N. B. Ho, C. Kimberling, and U. Larsson, 
\newblock Wythoff visions.
\newblock In U. Larsson, ed., {\em  Games of no chance}, Vol. 5, pp. 101--153. 
\newblock Cambridge University Press, 2017.

\bibitem{DuRi08} 
E. Duch\^{e}ne and M. Rigo, 
\newblock A morphic approach to combinatorial games: the Tribonacci case.
\newblock {\em RAIRO--Theoretical Informatics and Applications}, 42.2:375--383, 2008.

\bibitem{FR19} 
R. Fokkink and D. Rust,
\newblock A modification of Wythoff's Nim.
\newblock {arXiv:1904.08339v1, 2019}.

\bibitem{Guy91} 
R. K. Guy, ed.,
\newblock Combinatorial games.
\newblock American Mathematical Society, Proceedings of Symposia in Applied Mathematics, Vol. 43, 1991.

\bibitem{DEKVol4B} 
D. E. Knuth, 
\newblock The art of computer programming.
\newblock Addison-Wesley, Boston, Vol. 4B, Fascicle 5c, In preparation, 2019.
(See \url{http://www-cs-faculty.stanford.edu/~knuth/fasc5c.ps.gz}.)

\bibitem{LaWa14} 
U. Larsson and J. W\"{a}stlund,
\newblock Maharaja Nim: Wythoff's queen meets the knight.
\newblock {\em Integers: Electronic Journal of Combinatorial Number Theory}, 14\#G05, 2014.

\bibitem{Loth83} 
M. Lothaire,
\newblock {\em Combinatorics on words}.
\newblock Cambridge University Press, Encyclopedia of mathematics and its applications, Vol. 17, 1983.

\bibitem{MoSh14} 
H. Mousavi and J. Shallit, 
\newblock Mechanical proofs of properties of the Tribonacci word.
\newblock In  F. Manea and D. Nowotka, eds., 
{\em Combinatorics on Words: WORDS 2015}.
\newblock {\em Lecture Notes in Computer Science}, Vol. 9304. Springer, 2015, pp. 170--190.

\bibitem{OEIS} 
The OEIS Foundation Inc.,
\newblock {\em The On-Line Encyclopedia of Integer Sequences}.
\newblock Published electronically at \url{https://oeis.org}.

\bibitem{TaWe07} 
B. Tan and Z.-Y. Wen, 
\newblock Some properties of the Tribonacci sequence.
\newblock {\em European Journal of Combinatorics} 28.6:1703--1719, 2007.

\bibitem{VGL02} 
P. Vanderlind, R. K. Guy, and L. C. Larson,
\newblock {\em The inquisitive problem solver}.
\newblock The Mathematical Association of America, 2002.

\bibitem{Wyt07} 
W. A. Wythoff,
\newblock  A modification of the game of Nim.
\newblock {\em  Nieuw Arch. Wisk.}, 7:199--202, 1907.


\end{thebibliography}
\end{document}